\documentclass[12pt, a4paper]{article}

\usepackage{amssymb}
\usepackage{amsmath}
\usepackage{amsthm}
 \usepackage{dsfont}

  \newtheorem{theorem}{Theorem}
 \newtheorem{definition}{Definition}
  \newtheorem{corollary}{Corollary}
    \newtheorem{lemma}{Lemma}    
    \newtheorem{proposition}{Proposition}
        \newtheorem{question}{Question}

\DeclareMathOperator{\Div}{div}
\DeclareMathOperator{\Id}{Id}

\begin{document}

 \title{A remark on weak-strong uniqueness  for suitable weak solutions of the Navier--Stokes equations}
\author{Pierre Gilles Lemari\'e--Rieusset\footnote{LaMME, Univ Evry, CNRS, Universit\'e Paris-Saclay, 91025, Evry, France; e-mail : pierregilles.lemarierieusset@univ-evry.fr}}
\date{}\maketitle

\begin{abstract}
We extend Barker's  weak-strong uniqueness results for the Navier--Stokes equations and consider    a criterion involving Besov spaces and weighted Lebesgue spaces.  
\end{abstract}
 

 
 \section{The Prodi--Serrin criterion for weak-strong uniqueness}
 In this paper, we are interested in extensions of the Prodi--Serrin weak-strong uniqueness for (suitable) weak Leray solutions of the Navier--Stokes equations. We consider solutions of the Navier--Stokes equations
\begin{equation*}\left\{ \begin{split}&\partial_t \vec u+\vec u\cdot \vec\nabla \vec u=\Delta \vec u-\vec\nabla p\\&
\Div \vec u=0 
\\& \vec u(0,.)=\vec u_0\end{split}\right.\end{equation*} where $\vec u_0$ is a square-integrable divergence-free vector field on the space $\mathbb{R}^3$. 
 
Looking for weak solutions, where the derivatives are taken in the sense of distributions, it is better to write the first line of the system as
$$ \partial_t\vec u+ \Div\,(\vec u\otimes\vec u)=\Delta \vec u-\vec\nabla p.$$ If $\vec u$ is a solution on $(0,T)\times\mathbb{R}^3$ such that $\vec u\in L^\infty((0,T),L^2)$, then the pressure $p$ can be eliminated through the formula
 $$ \Div \,(\vec u\otimes\vec u)+\vec \nabla p= \mathbb{P} (\Div \,(\vec u\otimes\vec u))$$ where $\mathbb{P}$ is the Leray projection operator on solenoidal vector fields:
 $$\mathbb{P}\vec f=-\frac 1{\Delta}\vec\nabla\wedge(\vec\nabla\wedge\vec f).$$ 
 Moreover, $\vec u$ can be represented as a distribution which depends continuously on the time $t$ \cite{Lem16} as
 $$ \vec u=\vec u_0+\int_0^t \Delta\vec u-\mathbb{P}(\Div \,(\vec u\otimes\vec u))\, ds.$$

 Leray \cite{Ler34} proved existence of solutions $\vec u$ on $(0,+\infty)\times\mathbb{R}^3$ such that :
 \begin{itemize}
 \item $\vec u\in L^\infty_t L^2_x\cap L^2_t \dot H^1_x$
 \item $\lim_{t\rightarrow 0^+} \|\vec u(t,.)-\vec u_0\|_2=0$
 \item we have the Leray energy inequality
\begin{equation}\label{eqler} \|\vec u(t,.)\|_2^2+2\int_0^t \|\vec\nabla\otimes\vec u\|_2^2\, ds\leq \|\vec u_0\|_2^2\end{equation}
 \end{itemize} Such solutions are called \emph{Leray solutions}\footnote{Remark  that the continuity at $t=0$ of $t\mapsto \vec u(t,.)$  in $L^2$ norm is a consequence of the Leray inequality (\ref{eqler})}.
His proof is based on a compactness criterion, it does not provide any clue on the  uniqueness of the solution to the Cauchy initial value problem.

 A classical case of uniqueness of Leray weak solutions is the weak-strong uniqueness criterion described by Prodi and Serrin \cite{Pro59, Ser63}:
if $\vec u_0\in L^2$ and if the Navier-Stokes equations have a solution $\vec u$ on $(0,T)$  such that $$\vec u\in   L^p_t L^q_x\text{ with } \frac 2 p+\frac 3 q\leq1 \text{ and } 2\leq p\leq +\infty$$ then, if $\vec v$ is a Leray solution with the same initial value $\vec u_0$, we have $\vec u=\vec v$ on $(0,T)$. Let us remark that the existence of such a solution $\vec u$ restricts the range of the initial value $\vec u_0$: as a matter of fact, when $2<p<+\infty$, existence of  a time $T>0$ and of a solution $\vec u\in L^p_t L^q_x$ is equivalent  to the fact that $\vec u_0$ belongs to the Besov space $B^{-\frac 2 p}_{q,p}$ (see Theorem \ref{equivser} below). 

We will see that a corollary of Barker's theorem \cite{Bar18} shows the following extension of the criterion:  if $\vec u_0\in L^2$ and if the Navier-Stokes equations have a solution $\vec u$ on $(0,T)$  such that $$\sup_{0<t<T} t^{\frac 1 p}\| \vec u\|_q<+\infty \text{ with }   \frac 2 p+\frac 3 q\leq 1 \text{ and } 2< p<+\infty$$ and with
$$\lim_{t\rightarrow 0}t^{\frac 1 p}\| \vec u\|_q=0 \text{ if }  \frac 2 p+\frac 3 q\leq 1 $$ then, if $\vec v$ is a Leray solution with the same initial value $\vec u_0$, we have $\vec u=\vec v$ on $(0,T)$. Let us remark again that the existence of such a time $T$ and such a solution $\vec u$ is equivalent  to the fact that $\vec u_0$ belongs to the Besov space $B^{-\frac 2 p}_{q,\infty}\cap {\rm bmo}^{-1}_0$ (see  Definition \ref{bmozero} and Theorem \ref{equivbar} below). 
 
The space ${\rm bmo}^{-1}$ was introduced in 2001  by Koch and Tataru  \cite{KoT01}  for the study of mild solutions to the Navier--Stokes problem.  Let us recall the characterization of ${\rm bmo}^{-1}$ through the heat kernel \cite{KoT01, Lem02}:
\begin{proposition}$\ $\\ 
For $0<T<\infty$, define
\[ \|\vec u\|_{X_T}= \sup_{0<t<T} \sqrt t \|\vec u(t,.\|_\infty + \sup_{0<t<T, x_0\in \mathbb{R}^3}  (t^{-3/2} \int_0^t \int_{B(x_0,\sqrt t)} \vert \vec u(s,y)\vert^2\, dy\, ds)^{1/2}.
\] Then $\vec u_0\in {\rm bmo}^{-1}$ if and only if $(e^{t\Delta}\vec u_0)_{0<t<T}\in X_T$ (with equivalence of the norms $\|\vec u_0\|_{{\rm bmo}^{-1}}$ and $\|e^{t\Delta}\vec u_0\|_{X_T}$.
\end{proposition}

Recall that  the differential Cauchy problem for  Navier--Stokes equations reads as
\begin{equation*}\left\{ \begin{split}&\partial_t \vec u+\vec u.\vec\nabla \vec u=\Delta \vec u-\vec\nabla p\\&
 \Div \vec u=0 
\\& \vec u(0,.)=\vec u_0\end{split}\right.\end{equation*}

Under reasonable assumptions, the problem is equivalent to the following integro-differential problem : 
\[ \vec u=e^{t\Delta}\vec u_0-B(\vec u,\vec u)(t,x)\] where
\begin{equation} \label{bilin} B(\vec u,\vec v)=\int_0^t e^{(t-s)\Delta} \mathbb{P}\Div (\vec u\otimes\vec v) \, ds\end{equation} and $\mathbb{P}$ is the Leray projection operator.  (See \cite{Lem02, Lem16}  for details).

Koch and Tataru's theorem is then the following one:

 \begin{theorem} \label{theoKT}
$\ $\\
 There exists $C_0$ (which does not depend on $T$)  such that, if $\vec u$ and $\vec v$ are defined on $(0,T)\times\mathbb{R}^3$,
then
\[ \|B(\vec u,\vec v)\|_{X_T}\leq C_0 \|\vec u\|_{X_T}\|\vec v\|_{X_T}.\]
\end{theorem}
 
 \begin{corollary}\label{kocht} $\ $
 \\
  Let $\vec u_0\in {\rm bmo}^{-1}$ with $\text{ div }\vec u_0=0$. If $\|e^{t\Delta}\vec u_0\|_{X_T}<\frac 1{4C_0}$, then the integral Navier--Stokes equations have a solution on $(0,T)$ such that $\|\vec u\|_{X_T}\leq 2\| e^{t\Delta}\vec u_0\|_{X_T}$.
 
    This is the  {unique} solution such that $\|\vec u\|_{X_T}\leq \frac 1 {2C_0}$.
\end{corollary}

The solution $\vec u$ can be computed through Picard iteration as the limit of $\vec U_n$, where $\vec U_0=e^{t\Delta}\vec u_0$ and $\vec U_{n+1}=e^{t\Delta}\vec u_0-B(\vec U_n,\vec U_n)$. In particular, we have, by induction,
  $$\|\vec U_{n+1}-\vec U_n\|_{X_T} \leq (4C_0 \|e^{t\Delta}\vec u_0\|_{X_T})^{n+1}  \|e^{t\Delta}\vec u_0\|_{X_T}.$$
  
  Thus, Corollary \ref{kocht} grants local existence of a solution for the Navier--Stokes equations when the initial value belongs to the space ${\rm bmo}^{-1}_0$:
  \begin{definition}\label{bmozero}$\ $\\  $\vec u_0\in {\rm bmo}^{-1}_0$ if $\vec u\in {\rm bmo}^{-1}$ and $\lim_{T\rightarrow 0} \|e^{t\Delta}\vec u_0\|_{X_T}=0$.
\end{definition}

We may now recall Barker's theorem \cite{Bar18}:
    \begin{theorem} \label{theobar} $\ $\\ Let $\vec u_0$ be a divergence-free vector field with $\vec u_0\in L^2$.  Assume moreover $$\vec u_0\in  {\rm bmo}^{-1}_0\cap    B^{-s}_{q,\infty} \text{ with } 3<q<+\infty \text{ and }s<1 -\frac 2 q$$ and let $\vec u$ be the mild solution  of the Navier--Stokes equations with initial value $\vec u_0$ such that   $\|\vec u\|_{X_T}\leq \frac 1 {2C_0} $. If  $\vec v$ is a weak Leray solution of the Navier--Stokes equations with the same initial value $\vec u_0$,     then $\vec u=\vec v$ on $(0,T)$. \end{theorem}
 Again, we remark that, if $0<s<1-\frac 2 q$ and if $\vec u_0\in {\rm bmo}^{-1}_0$, if $\vec u$ is the mild solution with $\|\vec u\|_{X_T}\leq \frac 1 {2C_0} $, then $\vec u_0\in B^{-s}_{q,\infty}$ is equivalent to $$\sup_{0<t<T} t^{s/2} \|\vec u(t,.)\|_q<+\infty.$$
 In the following theorems, we shall state the assumptions in terms of the mild solution $\vec u$ instead of the initial value $\vec  u_0$. In  Theorem \ref{solbesov}, we shall give the equivalence between the assumption on the solution $\vec u$ and the assumption on the initial value $\vec u_0$.
 
 We aim to generalize Barker's result to a larger class of mild solutions.  Barker's result is based on an interpolation lemma which states that, if $\vec u_0\in  {\rm bmo}^{-1}_0\cap   L^2\cap   B^{-s}_{q,\infty} $ { with } $3<q<+\infty $ { and }$ -s>-1+\frac 2 q$, then $\vec u_0\in [L^2,B^{-\delta}_{\infty,\infty}]_{\theta,\infty}$ for some $\theta\in (0,1)$ and some $\delta\in (0,1)$. (Those conditions are in a way equivalent, as we shall see in Corollary \ref{embed}.) Then the comparison between the Leray solution  $\vec v$ and the mild solution $\vec u$ is performed through an estimation of both $\|\vec u-\vec w_\epsilon\|_2$ and $\|\vec v-\vec w_\epsilon\|_2$, where $\vec w_\epsilon$ is the solution of the Navier--Stokes problem with initial value $\vec w_{0,\epsilon}$ such that $\|\vec w_{0,\epsilon}-\vec u_0\|_2\leq C_1 \epsilon^\theta$ and $\|\vec w_\epsilon\|_{B^{-\delta}_{\infty,\infty}}< C_1 \epsilon^{\theta-1}$  (with $C_1$ depending on $\vec u_0$ but not on $\epsilon$).
 
 Our idea is to replace the space $L^2$ by the larger space $L^2_w=L^2(w\, dx)$ with $w(x)=\frac 1{(1+\vert x\vert)^2}$, and use the interpolation space $ [L^2_w,B^{-\delta}_{\infty,\infty}]_{\theta,\infty}$ for some $\theta\in (0,1)$ and some $\delta\in (0,1)$. As we shall no longer deal with the $L^2$ norm, the Leray inequality on $\|\vec v\|_2$ will not be sufficient. Instead, we shall consider a stricter class of weak solutions, namely the suitable weak Leray solutions \cite{CKN}:
 
   \begin{definition}\label{suits}$\ $\\ A Leray solution is suitable on $(0,T)$  if it fulfills the local energy inequality:
there exists    a non-negative locally finite measure $\mu$ on  $(0,T)\times\mathbb{R}^3$ such that we have
\begin{equation}\label{eneq} \partial_t(\vert\vec u\vert^2)+ 2 \vert\vec\nabla\otimes\vec u\vert^2=\Delta(\vert\vec u\vert^2)- \Div((2p+\vert\vec u\vert^2)\vec u)-\mu.\end{equation}
\end{definition}

   We may now state our main results. The first one (stated in \cite{Lem19}) weakens the integrability requirement on the solution $\vec u$ from the Lebesgue space $L^q$ to the Morrey space $M^{p,q}$. Recall that the Morrey space $M^{p,q}$, $1<p\leq q<+\infty$, is defined by
$$ \|f\|_{M^{p,q}}=\sup_{x_0\in\mathbb{R}^3}\sup_{0<r\leq 1} r^{\frac 3 q-\frac 3 p} (\int_{B(x_0,r)} \vert f(x)\vert^p\, dx)^{\frac 1 p}<+\infty.$$ For $p=1$, one replaces the requirement $f\in L^p_{\rm loc}$ by the assumption that $f$ is a locally finite Borel measure $\mu$ with $$ \|f\|_{M^{1,q}}=\sup_{x_0\in\mathbb{R}^3}\sup_{0<r\leq 1} r^{\frac 3 q-  3 } \int_{B(x_0,r)}  d\vert\mu\vert(x) <+\infty.$$ For $1<p\leq +\infty$, we have the continuous embeddings $$ L^q\subset M^{q,q}\subset M^{p,q}\subset M^{1,q}.$$

The idea of considering Morrey spaces instead of Lebesgue spaces is quite natural. Indeed, in the direct proof of the Prodi--Serrin criterion, a key estimate is the inequality
$$ \int \vert u v \vert \vert\vec\nabla w\vert \, dx \leq C \|u\|_q \|v\|_2^{1-\theta} \|\vec \nabla v\|_2^{\theta} \|\vec \nabla w\|_2$$ for $0\leq \theta\leq 1$ and $ \frac 1 q =\frac {\theta}3 $. This inequality still holds when the $L^q$ norm is replaced by the norm in the homogeneous Morrey space $\dot M^{2,q}$ with $0<\theta<1$ and $\frac 1 q=\frac\theta 3$ \cite{Lem07}.

    \begin{theorem} \label{theomorr} $\ $\\ Let $\vec u_0$ be a divergence-free vector field with $\vec u_0\in L^2\cap   {\rm bmo}^{-1}_0$.  Assume moreover that the mild solution $\vec u$  of the Navier--Stokes equations with initial value $\vec u_0$ such that   $\|\vec u\|_{X_T}\leq \frac 1 {2C_0} $ is such that
    $$ \sup_{0<t<T}  t^{s/2} \|\vec u(t,.)\|_{\dot M^{p,q}}<+\infty  \text{ with } 2<p\leq q<+\infty \text{ and }0\leq s<1-\frac 2 p.$$  If  $\vec v$ is a suitable  weak Leray solution of the Navier--Stokes equations with the same initial value $\vec u_0$,     then $\vec u=\vec v$ on $(0,T)$. \end{theorem}
    
    Let us remark that the statement and proof of Theorem \ref{theomorr} we gave  in \cite{Lem19} was false (we assumed only that $s<1-\frac 2 q$)\footnote{The mistake was due to an incorrect  equality $\rho=\eta\gamma$ while it should have been $\gamma=\eta\rho$; as $\eta<1$, the equality turned to be incorrect.}.

The second one  weakens the integrability requirement on the solution $\vec u$ from the Lebesgue space $L^q$ to the weighted Lebesgue  space $L^q(\frac 1{(1+\vert x\vert)^N}\, dx)$ for some $N\geq 0$.
    \begin{theorem} \label{theoweight} $\ $\\ Let $\vec u_0$ be a divergence-free vector field with $\vec u_0\in L^2\cap   {\rm bmo}^{-1}_0$.  Assume moreover that the mild solution $\vec u$  of the Navier--Stokes equations with initial value $\vec u_0$ such that   $\|\vec u\|_{X_T}\leq \frac 1 {2C_0} $ is such that
    $$ \sup_{0<t<T} t^{s/2} \|\vec u\|_{L^q(\frac 1{(1+\vert x\vert)^N}\, dx)}<+\infty \text{ with } N\geq 0 , 2<q<+\infty \text{ and }0\leq  s<1-\frac 2 q .$$ If  $\vec v$ is a suitable weak Leray solution of the Navier--Stokes equations with the same initial value $\vec u_0$,   then $\vec u=\vec v$ on $(0,T)$. \end{theorem}

Of course, Theorem \ref{theomorr} is a corollary of Theorem \ref{theoweight}, as $M^{p,q}\subset L^p(\frac 1{(1+\vert x\vert)^N}dx)$ for $N>3-\frac{3p}q $.\\

The paper is then organized in the following manner:
\begin{itemize} 
\item[$\bullet$] In Section 2, we define stable spaces and collect some technical results on generalized  Besov spaces based on stable spaces.
\item[$\bullet$] In Section 3, we define  potential spaces based on stable spaces and prove  some interpolation estimates.
\item[$\bullet$] In section 4, we give some remarks on the Koch and Tataru solutions for the Navier--Stokes problem.
\item[$\bullet$] In section 5, we study stability estimates for suitable weak Leray solutions with initial data in $L^2\cap [L^2(\frac 1{(1+\vert x\vert)^2},B^{-\delta}_{\infty,\infty}]_{\theta,\infty}$  (see Theorem \ref{stb}).
\item[$\bullet$] In section 6, we prove the uniqueness theorem  (Theorem \ref{theoweight}).
\item[$\bullet$] In section  7, we pay some further comments on Barker's conjecture on  the uniqueness problem.
\end{itemize}

%%%%%%%%

\section{Stable spaces and Besov spaces.}

We define the convolutor space $\mathbb{K}$ by  the following convention:
\begin{itemize} 
\item[$\bullet$] a {\bf suitable kernel} is a function $K\in L^1(\mathbb{R}^3)$ such that $K$ is radial and radially non-increasing (in particular, $K$ is nonnegative); this is noted as $K\in\mathbb{K}_0$
\item[$\bullet$] $f$ is a convolutor if $f\in L^1$ and if there exists $K\in \mathbb{K}_0$ such that $\vert f\vert\leq K$ almost everywhere
\item[$\bullet$] the norm of $f$ in $\mathbb{K}$ is defined as
$$ \|f\|_{\mathbb{K}} =\inf\{ \|K\|_1\ /\ K\in\mathbb{K}_0\text{ and } \vert f\vert\leq K \text{ a.e.}\}.$$
\end{itemize}

One easily checks that $\|\ \|_{\mathbb{K}}$ is a norm and that $(\mathbb{K},\|\ \|_\mathbb{K})$ is a Banach space.

\begin{definition}$\ $\\ A {\bf stable space} of measurable functions on $\mathbb{R}^3$ is a Banach space $E$ such that 
\begin{itemize}
\item[$\bullet$]  $E\subset L^1_{\rm loc}(\mathbb{R}^3)$
\item[$\bullet$]  if $f\in E$ and $g\in L^\infty$, $fg\in E$ and $\|fg\|_E\leq C \|f\|_E \|g\|_\infty$ (where $C$ does not depend on $f$ nor $g$)
\item[$\bullet$]  if $f\in E$ and $g\in \mathbb{K}$, $f*g\in E$ and $\|f*g\|_E\leq  C \|f\|_E \|g\|_{\mathbb{K}}$
(where $C$ does not depend on $f$ nor $g$).
\end{itemize}
\end{definition}

\noindent {\bf Examples of stable spaces}
\begin{enumerate} 
\item[a)]  $E=L^p$, $1\leq p\leq +\infty$.
\item[b)]  
 $E=L^p(w\,  dx)$ where $w$ belongs to the Muckenhoupt class $\mathcal{A}_p$ for some $1<p<+\infty$: if $g\in\mathbb{K}_0$, then $$\vert f*g(x)\vert\leq \|g\|_1 \mathcal{M}_f(x)$$ where $\mathcal{M}_f$ is the Hardy--Littlewood maximal function of $f$; recall that the Hardy--Littlewood maximal function is a bounded sublinear operator on $L^p(w\, dx)$ when $w\in\mathcal{A}_p$ \cite{Ste93}.
\item[c)] $E=L^p_{\rm uloc}$ for some $1\leq p\leq +\infty$, where 
$$ \|f\|_{L^p_{\rm uloc}} =\sup_{x_0\in\mathbb{R}^3} (\int_{B(x_0,1)} \vert f(x)\vert^p\, dx)^{\frac 1 p}.$$ By Minkowski's inequality, we have
$$ \|f*g\|_E\leq \int \vert g(y)\vert \|f(.-y)\|_{L^p_{\rm uloc}} \, dy=\|g\|_1 \|f\|_{L^p_{\rm uloc}}.$$
\item[d)] This example can be generalized to other shift-invariant spaces (for which the norms $\|f\|_E$ and $\|f(.-y)\|_E$ are equal). For instance, we may take $E$ as the Morrey space $M^{p,q}$, $1<p\leq q<+\infty$.\end{enumerate}

  Our next step is to introduce Besov-like Banach spaces based on stable spaces and to describe the regularity of Koch--Tataru solutions when the initial value belongs moreover to the Besov space.

\begin{definition}$\ $\\ Let $T\in (0,+\infty)$. Let $E$ be a  stable space of measurable functions on $\mathbb{R}^3$. For $s>0$ and $1\leq q\leq +\infty$, we define the Besov-like  Banach space $B^{-s}_{E, q}$ as the space of tempered distributions such that  
$$ t^{\frac s 2} \|e^{t\Delta}f\|_E\in L^q((0,T),\frac{dt}t).$$
 The norms $ \|  t^{\frac s 2} \|e^{t\Delta}f\|_E\|_{L^q((0,T),\frac{dt}t)}$ are all equivalent, so that $B^{-s}_{E, q}$ does not depend on $T$.
\end{definition}

  \begin{proof} Assume that  $t^{\frac s 2} \|e^{t\Delta}f\|_E\in L^q((0,T),\frac{dt}t)$ for some $T>0$ and consider $t\geq T$. We have 
  $$ e^{t\Delta} f= \frac{2}T \int_{T/2}^T e^{(t-\theta)\Delta} e^{\theta\Delta}f\, d\theta$$ so that
 \begin{equation*}\begin{split} \|e^{t\Delta}f\|_E \leq& C\frac{2}T \int_{T/2}^T \|e^{\theta\Delta}f\|_E\, d\theta \\\leq& C\frac{2}T  \|\theta^{s/2} \|e^{\theta\Delta}f\|_E\|_{L^q((0,T),\frac{d\theta}\theta)} \|\mathds{1}_{T/2<\theta}\theta^{1-\frac s2} \|_{L^{\frac q{q-1}}((0,T), \frac{d\theta}\theta)}.
 \end{split}\end{equation*} Equivalence of the norms is proved.  \end{proof}

  Remark: this proofs shows as well that, if $1\leq q\leq r\leq +\infty$, then $B^{-s}_{E,q}\subset B^{-s}_{E,r}$. Another obvious property of Besov spaces is that, if $0<s<\sigma$, then $B^{-s}_{E,\infty}\subset B^{-\sigma}_{E,1}$. \\

  The main result in this section is the following theorem:
  
  \begin{theorem}\label{solbesov}$\ $\\ Let $E$ be a  stable space of measurable functions on $\mathbb{R}^3$.
  Let $0<T<+\infty$, and let $\vec u_0\in {\rm bmo}^{-1}$ with $ {\rm  div \,}\vec u_0=0$ and  $\|e^{t\Delta}\vec u_0\|_{X_T}<\frac 1{4C_0}$. Let $\vec u$ be the  solution of  the integral Navier--Stokes equations    on $(0,T)$  such that $\|\vec u\|_{X_T}\leq \frac 1 {2C_0}$. Then the following assertions are equivalent for $0<\sigma<1$ and $2<q\leq +\infty$:\\ (A) $\vec u_0\in B^{-\sigma}_{E, q}$   
  \\ (B)  $ t^{\frac \sigma 2} \|\vec u\|_E\in L^q((0,T),\frac{dt}t)$.
  \end{theorem}
  \begin{proof} Let us remark that the operator $e^{(t-s)\Delta} \mathbb{P}{  \Div }$ is a matrix of convolution operators whose kernels are bounded by $ C \frac 1{(\sqrt{t-s}+\vert x-y\vert)^4}$, hence are controlled in the convolutor norm $\|\, \|_{\mathbb{K}}$ by $C \frac 1{\sqrt{t-s}}$. We thus have the inequality
\begin{equation*}\begin{split} \|B(\vec u,\vec v)\|_E\leq & C \int_0^t \frac 1{\sqrt{ t-s}} \|\vec u\otimes\vec v\|_E\, ds
\\ \leq& C' \sup_{0<s<t}\sqrt s \|\vec u(s,.)\|_\infty \int_0^t \frac 1{\sqrt{ t-s}} \frac 1{\sqrt s} \|\vec v(s,.) \|_E\, ds
 \end{split}\end{equation*} (and a similar estimate interchanging $\vec u$ and $\vec v$ in the last line).
 We thus want to estimate $J(t)= t^{-\frac 1 q+\frac\sigma 2} \int_0^t \frac 1{\sqrt{ t-s}} \frac 1{\sqrt s} s^{\frac 1 q-\frac\sigma 2} L(s)\, ds$ with $L\in L^q((0,T),\, dt)$. 
 \begin{itemize}
 \item[$\bullet$] if  $q=+\infty$, we easily check that $\|J\|_\infty\leq C_\sigma \|L\|_\infty$ (since $\sigma<1$).
 \item[$\bullet$] if  $\sigma\leq \frac 2 q$, we have $s^{\frac 1 q-\frac\sigma 2}\leq t^{\frac 1 q-\frac\sigma 2}$, so that $J(t)\leq \int_0^t\frac 1{\sqrt{t-s}} \frac 1{\sqrt s} L(s)\, ds$. If  $2<q<+\infty$, as $\frac 1{\sqrt s}$  belongs to the Lorentz space $L^{2,\infty}$, we use the product laws and convolution laws in Lorentz spaces to get  that, if $L\in L^q$,  $\frac 1 {\sqrt s}L\in L^{r,q}$ with $\frac 1 r= \frac 1 q+\frac 1 2$ and $\frac 1 {\sqrt s}*(\frac 1 {\sqrt s}L)\in L^{q,q}=L^q$. Thus, $\|J\|_q\leq C \|L\|_q$.
 \item[$\bullet$] if $\sigma> \frac 2 q$, we write
 $$ J(t)\leq C( \int_0^t \frac { (t-s)^{-\frac 1 q+\frac\sigma 2} }{\sqrt{ t-s}} \frac 1{\sqrt s} s^{\frac 1 q-\frac\sigma 2} L(s)\, ds+ \int_0^t \frac 1{\sqrt{ t-s}} \frac { s^{-\frac 1 q+\frac\sigma 2} }{\sqrt s} s^{\frac 1 q-\frac\sigma 2} L(s)\, ds)$$ and  we use again the product laws and convolution laws in Lorentz spaces to get  that, if $L\in L^q$,  $\frac 1 { s^{\frac {1+\sigma}2-\frac 1 q}}L\in L^{r,q}$ with $\frac 1 r= \frac {1+\sigma} 2  $ and
  $ \frac    1 { s^{\frac {1-\sigma}2+\frac 1 q}}*( \frac 1 { s^{\frac {1+\sigma}2-\frac 1 q}}L)\in L^{q,q}=L^q$.  We find again  $\|J\|_q\leq C \|L\|_q$.
 \end{itemize}
 We may now easily check that $(B)\implies (A)$ : we just write $e^{t\Delta}\vec u_0=\vec u+B(\vec u,\vec u)$ and
 $$ \left\| t^{\frac \sigma 2} \|B(\vec u,\vec u)\|_E\right\| _{L^q((0,T),\frac{dt}t)}\leq C \sup_{0<t<T} \sqrt t \|\vec u(t,.)\|_\infty \ \left\| t^{\frac \sigma 2} \| \vec u\|_E\right\| _{L^q((0,T),\frac{dt}t)}.
 $$
 In order to prove $(A)\implies (B)$, we write $\vec u$ as  the limit of $\vec U_n$, where $\vec U_0=e^{t\Delta}\vec u_0$ and $\vec U_{n+1}=e^{t\Delta}\vec u_0-B(\vec U_n,\vec U_n)$. By induction, $\vec U_n$ satisfies
 $$  \left\| t^{\frac \sigma 2} \| \vec U_n\|_E\right\| _{L^q((0,T),\frac{dt}t)}<+\infty$$ and 
 \begin{equation*}\begin{split}& \left\| t^{\frac \sigma 2} \| \vec U_{n+1}-\vec U_n\|_E\right\| _{L^q((0,T),\frac{dt}t)} \\& \leq C  \sup_{0<t<T} \sqrt t \|\vec U_n-\vec U_{n-1}\|_\infty  ( \left\| t^{\frac \sigma 2} \| \vec U_n\|_E\right\| _{L^q((0,T),\frac{dt}t)} +\left\| t^{\frac \sigma 2} \| \vec U_{n-1}\|_E\right\| _{L^q((0,T),\frac{dt}t)}) . \end{split}\end{equation*} 
If $$A_N= \left\| t^{\frac \sigma 2} \| \vec U_0\|_E\right\| _{L^q((0,T),\frac{dt}t)}+\sum_{n=0}^{N-1}  \left\| t^{\frac \sigma 2} \| \vec U_{n+1}-\vec U_n\|_E\right\| _{L^q((0,T),\frac{dt}t)}$$ and $\epsilon=4C_0\|\vec U_0\|_{X_T}$, we have 
$$ \left\| t^{\frac \sigma 2} \| \vec U_N\|_E\right\| _{L^q((0,T),\frac{dt}t)}\leq A_N$$ and
$$A_{N+1}\leq A_N(1+2C\epsilon^{N+1})\leq A_0 \prod_{j=1}^{N+1} (1+2C \epsilon^j).$$ This proves that $ \left\| t^{\frac \sigma 2} \| \vec u\|_E\right\| _{L^q((0,T),\frac{dt}t)}<+\infty$.
  \end{proof}
  
  Let us remark that the assumption $\vec u_0\in {\rm bmo}^{-1}$ can be dropped in some cases, as for example the solutions $\vec u$ in the Serrin class $L^q((0,T),L^r)$ with $\frac 2 q+\frac 3 r\leq 1$ and $3<r<+\infty$. In analogy with $L^r$, we define $r$-stable spaces in the following way:
  
  \begin{definition}$\ $\\
 For $2<r<+\infty$, a  {\bf $r$-stable space} of measurable functions on $\mathbb{R}^3$ is a stable space $E$ such that 
 \begin{itemize}
 \item[$\bullet$]   $E$ is contained in $B^{-\frac 3 r}_{\infty,\infty}$ and, for $f\in E$, $\|f\|_{B^{-\frac 3 r}_{\infty,\infty}}\leq C \|f\|_E$. 
  \item[$\bullet$] $E$ is contained in $L^2_{\rm loc}$
 \item[$\bullet$]  If $f,g\in E$ then $fg\in B^{-\frac 3 r}_{E,\infty}$ and $\|fg\|_{B^{-\frac 3 r}_{E,\infty}}\leq C\|f\|_E \|g\|_E$.
\end{itemize}
  \end{definition}
  
  The Morrey space $M^{2,r}$ is    a  $r$-stable space; it is more precisely the largest   $r$-stable space:
  \begin{lemma}$\ $\\
 Let $E$ be a  $r$-stable space of measurable functions on $\mathbb{R}^3$, where $r\in (2,+\infty)$.   Then 
  $E\subset M^{2,r}$ and $\|f\|_{M^{2,r}}\leq C\|f\|_E$.
 \end{lemma}
 \begin{proof}  Let $\rho<1$ and $x_0\in\mathbb{R}^3$. We have
 $$ e^{\rho^2\Delta}(f^2)(x_0)\geq \int_{B(x_0,\rho)} f^2(y) dy \inf_{y\in B(x_0,\rho)} W_{\rho^2}(x_0-y)= \frac  {e^{-\frac {1}{4}}}{(4\pi \rho^2)^{3/2}}  \int_{B(x_0,\rho)} f^2(y) dy$$
 where $W_t(x)=\frac 1{(4\pi t)^{3/2}} e^{-\frac {x^2}{4t}}$. On the other hand, we have
 $$ e^{\rho^2\Delta}(f^2)(x_0)\leq C \rho^{-\frac 3{r}} \|e^{\rho^2\Delta}(f^2)\|_{B^{-\frac 3 r}_{\infty,\infty}}\leq C'  \rho^{-\frac 3{r}} \|e^{\frac{\rho^2} 2\Delta}(f^2)\|_E\leq C'' \rho^{-\frac 6 r} \|f\|_E^2.
 $$ This gives
 $$ \int_{B(x_0,\rho)} f^2(y) dy \leq C \rho^{3-\frac 6 r} \|f\|_E^2$$ and thus $f\in M^{2,r}$.
 \end{proof}
  
   \begin{theorem}\label{equivser}$\ $\\ Let $E$ be a  $r$-stable space of measurable functions on $\mathbb{R}^3$. Let  $\vec u_0\in  E$ with $ {\rm  div \,}\vec u_0=0$. Let $0<\sigma<1$ and  $2<q<  +\infty$, with   $$  \frac 2 q\leq  \sigma\leq 1- \frac 3 r $$  and $q<+\infty$ if $\sigma=1-\frac 3 r$. Then the following assertions are equivalent:\\ (A) $\vec u_0\in B^{-\sigma}_{E, q}$   
  \\ (B)  There exists $T>0$ and a  solution  $\vec u$   of  the integral Navier--Stokes equations    on $(0,T)$ with initial value $\vec u_0$ such that  $ t^{\frac \sigma 2} \|\vec u\|_E\in L^q((0,T),\frac{dt}t)$.
  \end{theorem}
  
  (This theorem thus holds for solutions $\vec u\in L^q((0,T), E)$ under the Serrin condition $\frac 2 q+\frac 3 r\leq 1$.)
  \begin{proof} $ (A)\implies (B)$  is a direct consequence of Theorem \ref{solbesov} and of the embedding $B^{-\sigma}_{M^{2,r},q}\subset {\rm bmo}^{-1}_0$ for $\sigma\leq 1-\frac 3 r$ and $(\sigma,q)\neq(1-\frac 3 r,\infty)$. Indeed, we have, for $0<t<1$,
 \begin{equation*}\begin{split} \|e^{t\Delta}f\|_\infty \leq&  \frac{2}t \int_{t/2}^t \|e^{\theta\Delta}f\|_\infty\, d\theta 
 \\\leq&  C \frac{2}t  t^{-\frac 3{2r}}  \int_{t/2}^t \|e^{\frac\theta2\Delta}f\|_{M^{2,r}}\, d\theta
 \\\leq& C' t^{-1-\frac 3{2r}}   \|\theta^{\sigma/2} \|e^{\theta\Delta}f\|_{M^{2,r}}\|_{L^q((0,t), {d\theta} )}  \| \theta^{ -\frac \sigma2} \|_{L^{\frac q{q-1}}((t/2,t),  {d\theta} )}.
 \\ \leq& C'' t^{-1-\frac 3{2r}}   t^{\frac 1 q}  \|\theta^{\sigma/2} \|e^{\theta\Delta}f\|_{M^{2,r}}\|_{L^q((0,t), \frac{d\theta} \theta)} t^{1-\frac 1 q}t^{-\frac\sigma 2}
 \\ \leq& C''' t^{-1/2} t^{\frac{1-\sigma-3/r}2}   (\int_0^t (\theta^{\sigma/2} \|e^{\theta\Delta}f\|_{M^{2,r}})^q\frac{d\theta}\theta)^{1/q}
 \end{split}\end{equation*} 
 and 
 \begin{equation*}\begin{split} 
 \int_0^t \int_{B(x_0,\sqrt t)} \vert e^{s\Delta}f\vert^2\, dy\, ds\leq & C \int_0^t \|e^{s\Delta}f\|_{M^{2,r}}^2 t^{3/2-3/r}\, ds\\ \leq& C' t^{3/2-3/r} \| s^{\frac \sigma 2} \|e^{s\Delta}f\|_{M^{2,r}} \|_{L^q((0,t),\frac{ds}{s^\sigma})}^2 \|\mathds{1}\|_{L^{\frac q{q-2}}((0,t),\frac{ds}{s^\sigma})} 
 \\ \leq& C''  t^{3/2-3/r} t^{(1-\sigma)\frac 2 q}  \| s^{\frac \sigma 2} \|e^{s\Delta}f\|_{M^{2,r}} \|_{L^q((0,t),\frac{ds}{s})}^2  t^{(1-\sigma)(1-\frac 2 q)}
 \\ \leq& C'' t^{3/2} t^{1-\sigma-\frac 3 r}   (\int_0^t (s^{\sigma/2} \|e^{s\Delta}f\|_{M^{2,r}})^q\frac{ds}s)^{2/q}.
  \end{split}\end{equation*} 
  We now prove $(B)\implies (A)$. We use again the identity
  $$ e^{t\Delta}\vec u_0= \frac 2 t \int_{t/2}^t e^{(t-s)\Delta}e^{s\Delta}\vec u_0\, ds$$ and get
  $$ e^{2t\Delta}\vec u_0= \frac 2 t \int_{t/2}^t e^{(2t-s)\Delta}\vec u(s,.)\, ds+ \frac 2 t \int_{t/2}^t e^{(2t-s)\Delta}B(\vec u,\vec u)\, ds=\vec v(t,.)+\vec w(t,.).$$
  We want to estimate $\| t^{\sigma/2}\|e^{2t\Delta}\vec u_0\|_E\|_{L^q((0,T),\frac{dt}t)}=\|t^{\sigma/2-1/q} \|e^{2t\Delta}\vec u_0\|_E\|_{L^q((0,T),dt)}$.

 We have
  \begin{equation*}\begin{split}  t^{\sigma/2-1/q} \|\vec v(t,.)\|_E \leq& C t^{\sigma/2-1/q}\frac 2 t \int_{t/2}^t \| \vec u\|_E\, ds
  \\ \leq& C\frac 2 t  \int_{t/2}^t \| s^{\sigma/2-1/q} \|\vec u\|_E\| \, ds  \\ \leq& 4C \mathcal{M}_{s^{\sigma/2-1/q} \|\vec u\|_E}(t)
\end{split}\end{equation*}  and thus $t^{\sigma/2-1/q} \|\vec v(t,.)\|_E \in L^q((0,T),\ dt)$.

On the other hand, we have
 \begin{equation*}\begin{split} \|\vec w(t,.)\|_E \leq& \sup_{t/2\leq s\leq t} \|\int_0^s e^{( \frac{3 t}2-\tau)\Delta} \mathbb{P}{\rm div }\,  e^{\frac t2 \Delta}(\vec u\otimes\vec u)\, d\tau\|_E
 \\ \leq& C  \int_0^t \frac 1{\sqrt{\frac{3t}2-\tau}} \|e^{\frac t2 \Delta}(\vec u\otimes\vec u)\|_E\, d\tau
 \\ \leq& C'   \int_0^t    \frac {t^{\frac 1 2 - \sigma +\frac 1 q}}{( t-\tau)^{1-\sigma +\frac 1 q}}  \|e^{  t\Delta}\vert \vec u\vert^2\|_E  d\tau
 \\ \leq& C'' t^{\frac 1 2-\sigma +\frac 1 q-\frac3{2r} } \int_0^t   \frac 1{( t-\tau)^{1- \sigma +\frac 1 q}}  \|\vec u\|_E^2 d\tau
\end{split}\end{equation*}   and thus 
 \begin{equation*}\begin{split} t^{\sigma/2-1/q}  \|\vec w(t,.)\|_E \leq&  C  T^{\frac 1 2-\frac \sigma 2-\frac3{2r}}  \int_0^t   \frac 1{( t-\tau)^{1- \sigma +\frac 1 q}}     \|\vec u\|_E^2 d\tau
 \\ =& CT^{\frac 1 2-\frac\sigma 2-\frac3{2r} } \int_0^t  \frac 1{( t-\tau)^{1- \sigma +\frac 1 q}}  \tau^{-\sigma+\frac 2 q} (\tau^{\sigma/2-\frac 1 q} \|\vec u\|_E)^2 d\tau. 
\end{split}\end{equation*} 
If $J(\tau)= \tau^{\sigma/2-\frac 1 q} \|\vec u\|_E$, we have $J(\tau)\in L^q((0,T),d\tau)$, hence $J^2\in L^{q/2}((0,T),d\tau)$, $ \tau^{-\sigma +\frac 2 q} J^2\in L^{p_0,q/2}((0,T),\, dt)$ with $\frac 1 {p_0}= \frac 2 q+\sigma-\frac 2 q=\sigma$ and $\frac 1{  \tau^{1-\sigma+\frac 1 q}}* (\tau^{-\sigma+\frac 2 q} J^2)\in L^{p_1,q/2}((0,T),\, dt)$ with $\frac 1 {p_1}= \frac 1 {p_0}+ 1-\frac 1 \sigma+\frac 1 q-1  =\frac 1 q$.

Thus $t^{\sigma/2-1/q} \|e^{2t\Delta}\vec u_0(t,.)\|_E \in L^q((0,T),\ dt)$ and $\vec u_0\in B^{-\sigma}_{E,q}$.
\end{proof}

The case $(\sigma,q)=(1-\frac 3 r,+\infty)$ can be treated in a similar way:

   \begin{theorem}\label{equivbar}$\ $\\ Let $E$ be a  $r$-stable space of measurable functions on $\mathbb{R}^3$ with $3<r<+\infty$. Let  $\vec u_0\in  E$ with $ {\rm  div \,}\vec u_0=0$.  Then the following assertions are equivalent:\\ (A) $\vec u_0\in B^{-1+\frac 3 r}_{E, \infty}$   and $\lim_{t\rightarrow 0} t^{\frac 1 2-\frac 3{2r}} \|e^{t\Delta}\vec u_0\|_E=0$.
  \\ (B)  There exists $T>0$ and a  solution  $\vec u$   of  the integral Navier--Stokes equations    on $(0,T)$ with initial value $\vec u_0$ such that  $\sup_{0<t<T}  t^{\frac 1 2-\frac 3{2r}} \| \vec u \|_E<+\infty$
  and $\lim_{t\rightarrow 0} t^{\frac 1 2-\frac 3{2r}} \| \vec u \|_E=0$.

  \end{theorem}
   \noindent{\bf Remark:} We have the embedding 
   $B^{-1+\frac 3 r}_{E, \infty}\subset {\rm bmo}^{-1}$, 
   but this   does not grant existence of a solution. The extra condition $\lim_{t\rightarrow 0} t^{\frac 1 2-\frac 3{2r}} \|e^{t\Delta}\vec u_0\|_E=0$ is used to get $\vec u_0\in {\rm bmo}^{-1}_0$, and thus to have existence of a local solution.

  \section{Potential spaces and interpolation}
  
  If $E$ is a stable space, we define, for $s\in\mathbb{R}$, the potential space $H^s_E$ as $H^s_E=(\Id-\Delta)^{-s/2}E$, normed with $\|f\|_{H^s_E}=\|(\Id-\Delta)^{s/2}f\|_E$. For positive $s$, we have an obvious comparison of the potential space $H^{-s}_E$ with the Besov spaces:
 \begin{lemma} Let $E$ be  a stable space, and $s>0$. Then, $$ B^{-s}_{E,1}\subset H^{-s}_E \subset B^{-s}_{E,\infty}.$$\end{lemma}
 \begin{proof}
  Indeed, we have
  $$ (\Id-\Delta)^{-s/2} =\frac 1{\Gamma(s/2)} \int_0^{+\infty} e^{-t} e^{t\Delta} t^{s/2}\frac{dt}t.$$ If $f$ belongs to $B^{-s}_{E,1}$, then $ t^{s/2}\|e^{t\Delta}f\|_E\in L^1((0,1),\frac{dt}t)$ while $ \|e^\Delta f\|_1\leq   \|f\|_{B^{-s}_{E,\infty}}\leq C  \|f\|_{B^{-s}_{E,1}}$ so that
  $$ \|f\|_{H^{-s}_E} \leq  \frac 1{\Gamma(s/2)} (\int_0^1 t^{s/2} \|f \|_E \frac{dt}t + C \|e^{\Delta}f\|_E \int_0^{+\infty} e^{-t} t^{s/2}\frac{dt}t)\leq C' \|f\|_{B^{-s}_{E,1}} .$$
  Conversely, if $f\in H^{-s}_E $, $f=(\Id-\Delta)^{s/2} g$ and if $0<\theta<1$, then we pick $N\in\mathbb{N}$ with $N>s/2$ and write
 \begin{equation*}\begin{split}e^{\theta\Delta} f=&e^{\theta\Delta} (\Id-\Delta)^{N}   (\Id-\Delta)^{s/2-N} g\\=&\frac 1{\Gamma(N-s/2)}\int_0^{+\infty} e^{-t} (\Id-\Delta)^N e^{(t+\theta)\Delta}g \; t^{N-\frac s 2}\frac{dt}t. \end{split}\end{equation*}
 For $\alpha\in \mathbb{N}^3$, with $0\leq \vert \alpha\vert\leq 2N$, we have
 $$ \| \partial^\alpha e^{(t+\theta)\Delta}g \|_E\leq C_\alpha (t+\theta)^{-\frac {\vert \alpha\vert }2}\|g\|_E\leq C_\alpha (1+(t+\theta)^{-N})\|g\|_E$$ so that  
 \begin{equation*}\begin{split} \|e^{\theta\Delta} f\|_E \leq &C \|g\|_E \int_0^{+\infty} e^{-t}  (1+(t+\theta)^{-N}) \; t^{N-\frac s 2}\frac{dt}t\\ \leq& C \|g\|_E (\Gamma(N-s/2)+\int_0^\theta  t^{N-\frac s 2}\frac{dt}t+ \int_\theta^{+\infty}   \frac{dt}{t^{1+s/2}})\\ \leq& C' \|g\|_E \theta^{-s/2}.
 \end{split}\end{equation*} The lemma is proved.
 \end{proof}

 Let us recall the definition of Calder\'on's interpolation spaces  $[A_0,A_1]_\theta$ and $[A_0,A_1]^\theta$ \cite{Cal64}.   We assume that $A_0$ and $A_1$ are subspaces of $\mathcal{S}'$, so that $A_0\cap A_1$ and $A_0+A_1$ are well-defined.
 
 We begin with the definition of the first interpolate $[A_0,A_1]_\theta$.
 Let $\Omega$ be the open  complex strip $\Omega=\{z\in\mathbb{C}\ /\  0<\Re z<1\}$.   $\mathcal{F}(A_0,A_1)$ is the space of functions $F$  defined on the closed complex strip $\overline{\Omega}$ such that :
 \begin{enumerate}
 \item $F$ is continuous and bounded  from $\overline{\Omega}$ to $A_0+A_1$
 \item $F$ is analytic from $\Omega$ to $A_0+A_1$
 \item $t\mapsto F(it)$ is continuous from $\mathbb{R}$ to $A_0$, and $\lim_{\vert t\vert\rightarrow+\infty} \|F(it)\|_{A_0}=0$
  \item $t\mapsto F(1+it)$ is continuous from $\mathbb{R}$ to $A_1$, and $\lim_{\vert t\vert\rightarrow+\infty} \|F(1+it)\|_{A_0}=0$
 \end{enumerate}
 Then $$f\in [A_0,A_1]_\theta\Leftrightarrow \exists F\in\mathcal{F}(A_0,A_1) , f=F(\theta)$$ and
 $$ \|f\|_{[A_0,A_1]_\theta}=\inf_{f=F(\theta)} \max(\sup_{t\in\mathbb{R}}  \|F(it)\|_{A_0}, \sup_{t\in\mathbb{R}}  \|F(1+it)\|_{A_1}).$$
 
 Now, let us recall the definition of  the second interpolate $[A_0,A_1]^\theta$.   $\mathcal{G}(A_0,A_1)$ is the space of functions $G$  defined on the closed complex strip $\overline{\Omega}$ such that :
 \begin{enumerate}
 \item $\frac{1}{1+\vert z\vert} G$ is continuous and bounded  from $\overline{\Omega}$ to $A_0+A_1$
 \item $G$ is analytic from $\Omega$ to $A_0+A_1$
 \item $t\mapsto G(it)-G(0)$ is Lipschitz   from $\mathbb{R}$ to $A_0$ 
   \item $t\mapsto G(1+it)-G(1)$ is Lipschitz from $\mathbb{R}$ to $A_1$  \end{enumerate}
 Then $$f\in [A_0,A_1]^\theta\Leftrightarrow \exists G\in\mathcal{G}(A_0,A_1) , f=G'(\theta)$$ and
 $$ \|f\|_{[A_0,A_1]^\theta}=\inf_{f=G'(\theta)} \max(\sup_{t_1,t_2\in\mathbb{R}}  \|\frac{G(it_2)-G(it_1)}{t_2-t_1}\|_{A_0}, \sup_{t_1,t_2\in\mathbb{R}}  \|\frac{G(1+it_2)-G(1+it_1)}{t_2-t_1}\|_{A_1}).$$

Threeimportant properties of those complex interpolation functors are:
\begin{itemize}
\item[$\bullet$] the equivalence theorem: if $A_0$ (or $A_1$) is reflexive, then $[A_0,A_1]^\theta=[A_0,A_1]_\theta$ for $0<\theta<1$;
\item[$\bullet$] the duality theorem:  if $A_0\cap A_1$ is dense in $A_0$ and $A_1$, then $([A_0,A_1]_\theta)'=[A_0',A_1']^\theta$ for $0<\theta<1$.
\item[$\bullet$] the density theorem:  $A_0\cap A_1$ is dense in $[A_0,A_1]_\theta$ 
\end{itemize}

An easy classical example of interpolation concerns the Lebesgue spaces $L^p$ on a measured space $(X,\mu)$: $[L^{p_0}, L^{p_1}]_\theta=L^p$ with $1<p_0<+\infty$, $1<p_1<+\infty$, $0<\theta<1$ and $\frac 1 p= (1-\theta)\frac 1{p_0}+\theta \frac 1 {p_1}$. Indeed, if $f\in L^p$, we write $ f=F_\theta$ where $F_z(x)=\vert f(x)\vert^{(1-z)\frac p{p_0}+z\frac  p{p_1}} \frac{f(x)}{\vert f(x)\vert}$. If $p_0\leq p_1$, we have $\vert F_z(x)\vert\leq \vert f(x)\vert^{\frac p{p_0}}$   if $\vert f(x)\vert \geq 1$. and $\vert F_z(x)\vert\leq \vert f(x)\vert^{\frac p{p_1}}$ if $\vert f(x)\vert < 1$. By dominated convergence, this gives the continuity of $F$ from $\overline{\Omega}$ to $L^{p_0}+L^{p_1}$. For the holomorphy, we use the equivalence beetween (strong)  holomorphy and weak-* holomorphy; thus, it is enough to check that $z\in \Omega \mapsto \int F_z(x) g(x)\, d/mu$ is holomorph if $g\in L^{q_0}\cap L^{q_1}$, where $\frac 1 {q_i}+\frac 1{p_i}=1$. Thus, we obtain that $L^p\subset [L^{p_0}, L^{p_1}]_\theta$. As $$[L^{p_0},L^{p_1}]_\theta=[L^{p_0},L^{p_1}]^\theta =([L^{q_0},L^{q_1}]_\theta)'$$ and as $L^q$ is dense in $[L^{q_0},L^{q_1}]_\theta$ (where $\frac 1 q+\frac 1 p=1$), we obtain from the embedding $L^q\subset [L^{q_0}, L^{q_1}]_\theta$ that $[L^{p_0}, L^{p_1}]_\theta\subset L^p$.

A similar result holds for weighted Lebesgue spaces $L^p(w\, d\mu)$: $$[L^{p_0}(w_0\, d\mu), L^{p_1}(w_1\, d\mu)]_\theta=L^p(w\ d\mu)$$ with $1<p_0<+\infty$, $1<p_1<+\infty$, $0<\theta<1$ and $\frac 1 p= (1-\theta)\frac 1{p_0}+\theta \frac 1 {p_1}$ and $w=w_0^{1-\theta} w_1^\theta$. If $f\in L^p(w\ d\mu)$, one defines $$F_z(x)= \left( \frac{w(x)}{w_0(x)}\right)^{(1-z)\frac 1{p_0}}   \left( \frac{w(x)}{w_1(x)}\right)^{ z\frac 1{p_1}} \vert f(x)\vert^{(1-z)\frac p{p_0}+z\frac  p{p_1}} \frac{f(x)}{\vert f(x)\vert}.$$
We have $$\vert F_z(x)\vert\leq \max( \left(\frac{w(x) }{w_0(x)}\right)^{\frac 1{p_0}}\vert f(x)\vert^{\frac p{p_0}},  \left(\frac{w(x) }{w_1(x)}\right)^{\frac 1{p_1}}\vert f(x)\vert^{\frac p{p_1})}$$  The proof then is similar to the case of Lebesgue spaces.

If we want to interpolate Morrey spaces $M^{p_0,q_0}(\mathbb{R}^3)$ and $M^{p_1,q_1}(\mathbb{R}^3)$ and obtain a Morrrey space, then it is necessary to assume that $\frac{p_0}{q_0}=\frac{p_1}{q_1}$ \cite{Lem13, Lem14}. We then obtain:
$$ \left[ M^{p_0,q_0}, M^{p_1,q_1}\right]^\theta= M^{p,q}$$ when $1<p_0\leq q_0<+\infty$, $1<p_1\leq q_1<+\infty$, $\frac{p_0}{q_0}=\frac{p_1}{q_1}$, $0<\theta<1$, $\frac 1 p= (1-\theta)\frac 1{p_0}+\theta\frac 1{p_1}$ and $\frac 1 q=(1-\theta)\frac 1{q_0}+\theta\frac 1{q_1}$. As $M^{p_0,q_0}\cap M^{p_1,q_1}$ is not dense in $M^{p,q}$ and is dense in $ \left[ M^{p_0,q_0}, M^{p_1,q_1}\right]_\theta$, we can see that we must use the second interpolation functor. The embedding $  \left[ M^{p_0,q_0}, M^{p_1,q_1}\right]^\theta\subset M^{p,q}$ is obvious: for a ball $B$ with radius $r\leq 1$, we have that the map $f\mapsto f\mathds{1}_B$ is bounded from $M^{p_0,q_0}$ to $L^{p_0}$ with norm less or equal to $r^{3(\frac 1{p_0}-\frac 1{q_0})}$ and from $M^{p_1,q_1}$ to $L^{p_1}$ with norm less or equal to $r^{3(\frac 1{p_0}-\frac 1{q_0})}$, hence from  $  \left[ M^{p_0,q_0}, M^{p_1,q_1}\right]^\theta $ to $  \left[ L^{p_0}, L^{p_1}\right]^\theta$ with norm less or equal to $r^{3(\frac 1 p-\frac 1 q)}$. As  $\left[ L^{p_0}, L^{p_1}\right]^\theta=L^p$, we obtain the desired estimates.

If $f$ belongs to $M^{p,q}$, we define  $F_z(x)=\vert f(x)\vert^{(1-z)\frac p{p_0}+z\frac  p{p_1}} \frac{f(x)}{\vert f(x)\vert}$.  As $\vert F_z(x)\vert \leq \max( ( \vert f(x)\vert^{\frac p{p_0}}  \vert f(x)\vert^{\frac p{p_1}})$, 
we find that $z\mapsto F_z$ is bounded from $\overline{\Omega}$ to $M^{p_0,q_0}+M^{p_1,q_1}$ and holomorph on the open strip $\Omega$ (again by equivalence between analyticity and weak-* analyticity). 
But it is no longer continuous, and we cannot apply the first functor of Calder\'on. Instead, we follow Cwikel and Janson \cite{CwJ84} and define $G_z=\int_{1/2}^z F_w\, dw$. We may then apply the definition of the second functor and find that $f\in  \left[ M^{p_0,q_0}, M^{p_1,q_1}\right]^\theta$. Thus, $ \left[ M^{p_0,q_0}, M^{p_1,q_1}\right]^\theta= M^{p,q}$. 

 Now, we are going to describe complex interpolation of potential spaces on weighted Lebesgue spaces when varying both the regularity exponents and   the weights\footnote{This can be seen as a variation on Stein's interpolation theorem \cite{Ste56, CwJ84}.}:
 
 \begin{proposition}$\ $\\ Let $\theta\in (0,1)$, $s_0$, $s _1$ be real numbers, $1<p_0,p_1<+\infty$ and $s=(1-\theta)s_0+\theta s_1$ and $\frac 1 p=(1-\theta)\frac 1{p_0}+\theta \frac 1{p_1}$. Then, if  $w_0$ is a weight in the Muckenhoupt class $\mathcal{A}_{p_0}$ and $w_1$ is a weight in the Muckenhoupt class $\mathcal{A}_{p_1}$, 
 $$ (\Id-\Delta)^{-s} L^p(w_0^{1-\theta}w_1^\theta\, dx)= [(\Id-\Delta)^{-s_0} L^{p_0}(w_0\, dx), (\Id-\Delta)^{-s_1}L^{p_1}(w_1\, dx)]_\theta .$$
  \end{proposition}
 
 \begin{proof}
 Let $f= (\Id-\Delta)^{-s} g$ where $g\in L^p(w\ dx)$  We define
 $$ H_z(x)= \left( \frac{w(x)}{w_0(x)}\right)^{(1-z)\frac 1{p_0}}   \left( \frac{w(x)}{w_1(x)}\right)^{ z\frac 1{p_1}} \vert f(x)\vert^{(1-z)\frac p{p_0}+z\frac  p{p_1}} \frac{f(x)}{\vert f(x)\vert}$$ and 
 $$  F_{z,\epsilon}(.)=\left(\frac  {2-\theta}{2-z}\right)^4 e^{\epsilon\Delta}(\Id-\Delta)^{-(1-z)s_0-zs_1}   H_z  .$$
 We first remark that, for $\epsilon>0$ fixed,  the operators $e^{\epsilon\Delta}(\Id-\Delta))^{-\tau}$ with $\tau\in [s_0,s_1]$ are equicontinuous from $L^{p_i}(w_i\, dx)$ to $(\Id-\Delta)^{-s_i}L^{p_i}(w_i\, dx)$    (it is enough to check that the norms of the convolutors $e^{\epsilon\Delta}(\Id-\Delta)^{s_i-\tau}$ in $\mathbb{K}$ are uniformly bounded. 
 
 Moreover, the operators 
 $\left(\frac  {2-\theta}{2-it}\right)^4  (\Id-\Delta)^{-it}$, $t\in\mathbb{R}$, are uniformly bounded on $L^{p_i}(w_i\, dx)$. Let us recall the definition of Calder\'on--Zygmund convolutors.
 A Calder\'on--Zygmund convolutor is a distribution $K\in \mathcal{S}'(\mathbb{R}^3)$ such that $\hat K\in L^\infty$ (so that convolution with $K$ is a bounded operator on $L^2$) and such that, when restricted to $\mathbb{R}^3\setminus\{0\}$, $K$ is defined by a locally Lipschitz function such that
 $\sup_{x\neq 0} \vert x\vert^3 \vert K(x)\vert + \vert x\vert^4 \vert \vec \nabla x\vert <+\infty$. The space ${\rm CZ}$  of   Calder\'on--Zygmund convolutors is normed by
 $$ \|K\|_{CZ}=\|\hat K\|_\infty + \sup_{x\neq 0} \vert x\vert^3 \vert K(x)\vert + \vert x\vert^4 \vert \vec \nabla x\vert.$$  If $1<p<+\infty$ and  $w\in \mathcal{A}_p$ and $K\in {\rm CZ}$, we have $\|f*K\|_{L^p(w\, dx)} \leq C_{w,p}  \|f\|_{L^p(w\, dx)} \|K\|_{\rm CZ}$. 
 Since we have
 $$ \|K\|_{CZ}\leq C   \sum_{\vert \alpha\vert \leq 4} \|\vert \xi\vert^{\vert \alpha\vert } \partial^\alpha_\xi \hat K\|_\infty,$$ it is clear that 
 $\left(\frac  {2-\theta}{2-it}\right)^4  (\Id-\Delta)^{-it}f=K_t*f$ with $\sup_{t\in\mathbb{R}} \|K_t\|_{\rm CZ}<+\infty$.

 We may apply the second interpolation functor and find that $e^{\epsilon\Delta}f=F_{\theta,\epsilon}\in [(\Id-\Delta)^{-s_0} L^{p_0}(w_0\, dx), (\Id-\Delta)^{-s_1}L^{p_1}(w_1\, dx)]_\theta$ if $g\in L^p(w\, dx)$. Moreover its norm is controlled independently from $\epsilon>0$ as, for $\alpha=0$ or $\alpha=1$, the functions $H_{\alpha+it}$ are bounded in $L^{p_\alpha}(w_\alpha\, dx)$,   the operators $\left(\frac  {2-\theta}{2-it}\right)^4  (\Id-\Delta)^{-it}$ are equicontinuous  on   $L^{p_\alpha}(w_\alpha\, dx)$  and the operators $^{\epsilon \Delta}$ are   equicontinuous  on   $L^{p_\alpha}(w_\alpha\, dx)$. One then writes $$F_{\alpha+it,\epsilon}=(\Id-\Delta)^{-(1-\alpha)s_0-\alpha s_1} \left(e^{\epsilon\Delta} \left(\frac  {2-\theta}{2-it}\right)^4  (\Id-\Delta)^{-it} H_{\alpha+it}\right).$$
 
 To conclude, we remark that $L^{p_i}(w_i \, dx)$ is the dual of $L^{q_i}(w^{-\frac{q_i}{p_i}}\, dx)$ and that $\mathcal{S}$ is dense in this predual.    Thus $e^{\epsilon\Delta}f$ is bounded in 
 $$[H^{s_0}_{L^{p_0}(w_0\, dx)},  H^{s_1}_{L^{p_1}(w_1\, dx)}]^\theta=([(\Id-\Delta)^{s_0} L^{q_0}(w_0^{-\frac{q_0}{p_0}}\, dx), (\Id-\Delta)^{s_1}L^{q_1}(w_1^{-\frac{q_1}{p_1}}\, dx)]_\theta)'$$ if $(\Id-\Delta)^{s}f\in L^p(w\, dx)$. As $\epsilon$ goes to $0$, $e^{\epsilon\Delta}f$ is weak-* convergent to $f$. Thus, 
 $$(\Id-\Delta)^{-s} L^p(w\, dx)\subset   [(\Id-\Delta)^{-s_0} L^{p_0}(w_0\, dx), (\Id-\Delta)^{-s_1}L^{p_1}(w_1\, dx)]^\theta$$
 and we can switch the second and the first interpolation functors as $H^{s_0}_{L^{p_0}(w_0\, dx)}$ is reflexive.

 Conversely, assume that $f\in  [(\Id-\Delta)^{-s_0} L^{p_0}(w_0\, dx), (\Id-\Delta)^{-s_1}L^{p_1}(w_1\, dx)]^\theta$ and pick $F\in \mathcal{F}((\Id-\Delta)^{-s_0} L^{p_0}(w_0\, dx), (\Id-\Delta)^{-s_1}L^{p_1}(w_1\, dx)))$ such that $ f=F(\theta)$. Define $$H_{z,\epsilon} = \left(\frac  {2-\theta}{2-z}\right)^4 e^{\epsilon\Delta}(\Id-\Delta)^{(1-z)s_0+zs_1}   F_z  .$$ We easily check that $H_{z,\epsilon}\in \mathcal{A}(L^{p_0}(w_0\, dx), L^{p_1}(w_1\, dx))$ with $H_{\theta,\epsilon}=e^{\epsilon\Delta} (\Id-\Delta)^s f$.
Thus, we find that $e^{\epsilon\Delta} (\Id-\Delta)^s f$ is bounded in $[L^{p_0}(w_0\, dx), L^{p_1}(w_1\, dx]_\theta=L^p(w\, dx)$, and finally $f\in (\Id-\Delta)^{-s} L^p(w\, dx)$.
 \end{proof}
 
 \begin{corollary}\label{embed}$\ $\\  Let $2<q<+\infty$ and $s<1-\frac 2 q$. Then there exists such that :
 \\ a) There exists  $\gamma>0$ and $2<r<+\infty$ such that $\gamma+\frac 3 r<1$  and $\theta\in (0,1)$ such that $$ B^{-s}_{q,\infty}\subset [L^2,H^{-\gamma}_r]_{\theta,\infty}\subset [L^2, B^{-\sigma}_{\infty,\infty}]_{\theta,\infty}.$$
b)  For $0\leq N<\frac 4 q$, there exists  $\gamma>0$ and $2<r<+\infty$ such that $\gamma+\frac 3 r<1$  and  and $\theta\in (0,1)$ such that $$  B^{-s}_{L^q(1+\vert x\vert^{-N}dx),\infty}\subset [L^2(\frac{dx}{(1+\vert x\vert)^2}), H^{-\gamma}_r]_{\theta,\infty}\subset [L^2(\frac{dx}{(1+\vert x\vert)^2}), B^{-\sigma}_{\infty,\infty}]_{\theta,\infty}.$$
 \end{corollary}
 
 \begin{proof} If $s<\sigma<1-\frac 2 q$, we have $B^{-s}_{q,\infty}\subset H^{-\sigma}_{L^q}$  and   $B^{-s}_{L^q(1+\vert x\vert^{-N}dx),\infty}\subset   H^{-\sigma}_{L^q(1+\vert x\vert^{-N}dx)}$. Thus, if $r>q$, we have, for $\theta\in (0,1)$, $\gamma>\sigma$   such that $(1-\theta) \frac 1 2 +\theta \frac 1 r=\frac 1 q$ and $\theta \gamma=\sigma$, $$B^{-s}_{q,\infty}\subset [L^2, H^{-\gamma}_r]_\theta\subset [L_2, H^{-\gamma}_r]_{\theta,\infty}\subset [L_2, B^{-\gamma-\frac 3 r}_{\infty,\infty}]_{\theta,\infty}.$$ As $\gamma +\frac 3 r=  (1-\frac 2 r)\frac\sigma{1-\frac 2 q} +\frac 3 r=\frac\sigma{1-\frac 2 q} +O(\frac 1 r) $, we have $\gamma +\frac 3 r<1$ for $r$ large enough. Similarly, if $(1-\theta) M= N$ and $M<2$ (so that in particular $\frac 1{(1+\vert x\vert)^M}\in \mathcal{A}_2$), we have
\begin{equation*}\begin{split}B^{-s}_{L^q(\frac 1{(1+\vert x\vert)^N}dx),\infty}\subset &[L^2(\frac 1{(1+\vert x\vert)^M}dx), H^{-\gamma}_r]_\theta\\\subset &[L_2(\frac 1{(1+\vert x\vert)^2}dx), H^{-\gamma}_r]_{\theta,\infty}\subset [L_2(\frac 1{(1+\vert x\vert)^2}dx), B^{-\gamma-\frac 3 r}_{\infty,\infty}]_{\theta,\infty}.\end{split}\end{equation*}
 As $M=\frac N{1-\theta}=N \frac{\frac1 2-\frac 1 r}{\frac 1 q-\frac 1 r}= \frac{qN}2 +O(\frac 1 r)$, we have $M<2$ for $r$ large enough.
 \end{proof}
%%%%%%%%%%%%%%
  \section{Mild solutions for the Navier--Stokes equation.}
  In this section, we develop some remarks on the solutions provided by Koch and Tataru's theorem  (Theorem \ref{theoKT} and Corollary \ref{kocht}).    
  
    Let $\vec u_0\in {\rm bmo}^{-1}_0$ with $\text{ div }\vec u_0=0$. If $\|e^{t\Delta}\vec u_0\|_{X_T}<\frac 1{4C_0}$, then the integral Navier--Stokes equations have a solution on $(0,T)$ such that $\|\vec u\|_{X_T}\leq 2\| e^{t\Delta}\vec u_0\|_{X_T}$. This solution is computed   through Picard iteration as the limit of $\vec U_n$, where $\vec U_0=e^{t\Delta}\vec u_0$ and $\vec U_{n+1}=e^{t\Delta}\vec u_0-B(\vec U_n,\vec U_n)$. In particular, we have, by induction,
  $$\|\vec U_{n+1}-\vec U_n\|_{X_T} \leq (4C_0 \|e^{t\Delta}\vec u_0\|_{X_T})^{n+1}  \|e^{t\Delta}\vec u_0\|_{X_T}$$ and $$ \|\vec U_n\|_{X_T}\leq   2\| e^{t\Delta}\vec u_0\|_{X_T}.$$
  
  It is easy to check that $\vec u$ is smooth: if $X_\alpha=L^\infty$ if $\alpha=0$ and $\dot B^{\alpha}_{\infty,\infty}$ if $\alpha>0$, we have
  $$ \|uv\|_{X_\alpha}\leq C_\alpha (\|u\|_\infty \|v\|_{X_\alpha}+\|v\|_\infty \|u\|_{X_\alpha})$$ and $\|\vec u\|_\infty\leq 2 \|e^{t\Delta}\vec u_0\|_{X_T} \frac 1{\sqrt t}$ for $0<t<T$ while
  $$ \vec u(t,.)=e^{\frac t 2 \Delta}\vec u(t/2,.)-\int_0^{t/2} e^{(\frac t 2-s)\Delta} \mathbb{P}\Div(\vec u(\frac t 2+s,.)\otimes \vec u(\frac t 2+s,.))\, ds$$ so that
  $$ \|\vec u(t,.)\|_{X_{(n+1)/2}}\leq C \frac 1{t^{1/4}} \|\vec u(t/2,.)\|_{X_{n/2}} +  C  \|e^{t\Delta}\vec u_0\|_{X_T}  \int_0^{t/2} \frac 1{(\frac t 2-s)^{3/4}} \frac 1{\sqrt s} \|\vec u(\frac t 2+s, .)\|_{X_{n/2}}\, ds$$ and, by induction on $n$,
  $$ \|\vec u( t, .)\|_{X_{n/2}}\leq C_n t^{-\frac 1 2-\frac n 4}.$$
  Thus, for $0<t<T$, $\vec u$ is smooth with respect to the space variable $x$. So is $\vec\nabla p$, by hypoellipticity of the Laplacian (as $\Delta p=-\sum_{i=1}^3\sum_{j=1}^3 \partial_i u_j \partial_j u_i$). Then we have smoothness with respect to the time variable by controlling  the time derivatives through the Navier--Stokes equations.

  \begin{proposition}\label{propweight}$\ $\\ Let $\vec u_0\in {\rm bmo}^{-1}_0$ with $\text{ div }\vec u_0=0$. Let $E\subset \mathcal{S}'$ be a stable space.
  If moreover $\vec u_0$ belongs to  $E$, then the small solution $\vec u$ to  the integral Navier--Stokes equations with initial value $\vec u_0$, i.e.  the solution on $(0,T)$ such that $\|\vec u\|_{X_T}\leq 2\| e^{t\Delta}\vec u_0\|_{X_T}$, satisfies $\sup_{0<t<T} \|\vec u(t,.)\|_E<+\infty$ and $\lim_{t\rightarrow 0} \|\vec u(t,.)-e^{t\Delta}\vec u_0\|_E=0$. In particular, if $\mathcal{S}$ is dense in $E$, then 
  $\lim_{t\rightarrow 0} \|\vec u(t,.) -\vec u_0\|_E=0$.
  
  Moreover, if  $E\subset \mathcal{S}'$ is the dual of a space $E_0$ where $\mathcal{S}$ is dense, $$\sup_{0<t<T} \sqrt t \|\vec\nabla\otimes \vec u\|_E<+\infty.$$\end{proposition}

   \begin{proof} We have 
  $$ \|B(\vec u,\vec v)(t,.)\|_E\leq C_E \int_0^t \frac 1{\sqrt{t-s}} \min( \|\vec u\|_\infty \| \vec v\|_E, \|\vec u\|_E\|\vec v\|_\infty) \, ds.$$                                                                                                                                                                                                                                                                                                                                                                                                                                                                                                                                                                                                                                                                                                                                                                                                                                                                                                                                                                                                                                                                                                                                                                                                                                                                                                                                                                                                                                                                                                                                                                                                                                                                                                                                                                                                                                                                                                                                                                                                                                                                                                                                                                                                                                                                                                                                                                                                                                                                                                                                                                                                                                                                                                                                                                                                                                                                                                                                                                                                                                                                                                                                                                                                                                                                                                                                                                                                                                                                                                                                                                                                                                                                                                                                                                                                                                                                                                                                                                                                                                                                                                                                                                                                                                                                                                                                                                                                                                                                                                                                                                                                                                                                                                                                                                                                                                                                                                                                                                                                                                                                                                                                                                                                                                                                                                                                                                                                                                                                                                                                                                                                                                                                                                                                                                                                                                                                                                                                                                                                                                                                                                                                                                                                                                                                                                                                                                                                                                                                                                                                                                                                                                                                                                                                                                                                                                                                                                                                                                                                                                                                                                                                                                                                                                                                                                                                                                                                                                                                                                                                                                                                                                                                                                                                                                                                                                                                                                                                                                                                                                                                                                                                                                                                                                                                                                                                                                                                                                                                                                                                                                                                                                                                                                                                                                                                                                                                                                                                                                                                                                                                                                                                                                                                                                                                                                                                                                                                                                                                                                                                                                                                                                                                                                                                                                                                                                                                                                                                                                                                                                                                                                                                                                                                                                                                                                                                                                                                                                                                                                                                                                                                                                                                                                                                                                                                                                                                                                                                                                                                                                                                                                                                                                                                                                                                                                                                                                                                                                                                                                                                                                                                                                                                                                                                                                                                                                                                                                                                                                                                                                                                                                                                                                                                                                                                                                                                                                                                                                                                                                                                                                                                                                                                                                                                                                                                                                                                                                                                                                                                                                                                                                                                                                                                                                                                                                                                                                                                                                                                                                                                                                                                                                                                                                                                                                                                                                                                                                                                                                                                                                                                                                                                                                                                                                                                                                                                                                                                                                                                                                                                                                                                                                                                                                   By induction we have $\vec U_n \in L^\infty((0,T),E)$ with, for $n\geq 0$ [and $\vec U_{-1}=0$]
\begin{equation*}\begin{split} &\|\vec U_{n+1}(t,.)-\vec U_n(t,.)\|_E\\\leq& C \int_0^t \frac 1{\sqrt{t-s}} \frac 1{ \sqrt s }\sqrt s \|\vec U_n(s,.)-\vec U_{n-1}(s,.)\|_\infty (\|\vec U_n(s,.)\|_E+\|\vec U_{n-1}(s,.)\|_E)\, ds
\\\leq& C' (4C_0 \|e^{t\Delta}\vec u_0\|_{X_T})^n  \sum_{k=0}^n \|\vec U_k-\vec U_{k-1}\|_{L^\infty((0,T), E)}.
\end{split}\end{equation*}
Thus, we have
$$\sum_{k=0}^{+\infty} \|\vec U_k-\vec U_{k-1}\|_{L^\infty((0,T), E)} \leq \|\vec U_0\|_{L^\infty((0,T), E)} \prod_{n=0}^\infty (1+  C (4C_0 \|e^{t\Delta}\vec u_0\|_{X_T})^n).  $$
Thus, $\sup_{0<t<T} \|\vec u(t,.)\|_E<+\infty$.

We have  $\sup_{t>0} \sqrt t \|\vec\nabla\otimes \vec U_0\|_E<+\infty$. We will  show by induction that $\sup_{t>0} \sqrt t \|\vec\nabla\otimes \vec U_n\|_E<+\infty$. Indeed, for $\eta\in (0,1)$ and $0<t<T$, we have
 $$ \vec U_{n+1}(t,.)=e^{\eta t \Delta}\vec U_{n+1}((1-\eta)t,.)-\int_0^{\eta t} e^{( \eta t-s)\Delta} \mathbb{P}\Div(\vec U_n( (1-\eta)t+s,.)\otimes \vec U_n((1-\eta)t+s,.))\, ds$$ 
and, 
 since  $\Div (\vec u\otimes\vec v)=\vec u\cdot\vec\nabla \vec v$, 
 $$\partial_j \vec U_{n+1}(t,.)=e^{\eta t \Delta}\vec U_{n+1}((1-\eta)t,.)-\int_0^{\eta t} e^{( \eta t-s)\Delta} \mathbb{P}\partial_j(\vec U_n( (1-\eta)t+s,.)\cdot\vec\nabla \vec U_n((1-\eta)t+s,.))\, ds.$$ 
 This gives
 \begin{equation*}\begin{split} & \|\vec \nabla  \vec U_{n+1}(t,.)\|_E
 \\\leq & C \frac 1{\sqrt {\eta t}} \|\vec U_{n+1}\|_{L^\infty((0,T),E)} \\&+ C \int_0^{\eta t} \frac 1{\sqrt{\eta t-s}} \frac 1{(1-\eta)t+s}\, ds   \sup_{0<s<T} \sqrt s \|\vec\nabla \otimes \vec U_n(s,.)\|_E \sqrt s \| \vec U_{n}(s,.)\|_\infty
  \\\leq & C_1 \frac 1{\sqrt {\eta t}}  + C_1 \frac{\sqrt \eta}{1-\eta}   \frac 1{\sqrt t}  \sup_{0<s<T} \sqrt s \|\vec\nabla \otimes \vec U_n(s,.)\|_E
 \end{split}\end{equation*} where $C_1$ does not depend on $n$ nor on $\eta$. For $\eta$ small enough, we have $C_1 \frac{\sqrt \eta}{1-\eta} <\frac 1 4$ and  $\sup_{0<s<T} \sqrt s \|\vec\nabla \otimes \vec U_0(s,.)\|_E\leq  2C_1\frac 1{\sqrt \eta}$. By induction, we get   $\sup_{0<s<T} \sqrt s \|\vec\nabla \otimes \vec U_n(s,.)\|_E\leq  2C_1\frac 1{\sqrt \eta}$ for every $n\in\mathbb{N}$. If $E\subset \mathcal{S}'$ is the dual of a space $E_0$ where $\mathcal{S}$ is dense, we conclude that $\sup_{0<s<T} \sqrt s \|\vec\nabla \otimes \vec u(s,.)\|_E<+\infty$.
 \end{proof}

  \begin{proposition}\label{propenerg}$\ $\\ Let $\vec u_0\in {\rm bmo}^{-1}_0$ with $\text{ div }\vec u_0=0$. Let $w=\frac 1{(1+\vert x\vert)^N}$ where $0\leq N<3$.  
  If moreover $\vec u_0$ belongs to  $L^2(w\, dx))$, then the small solution $\vec u$ to  the integral Navier--Stokes equations with initial value $\vec u_0$, i.e.  the solution on $(0,T)$ such that $\|\vec u\|_{X_T}\leq 2\| e^{t\Delta}\vec u_0\|_{X_T}$, satisfies $\vec u\in L^\infty((0,T),L^2(w\, dx))$ and $\vec\nabla\otimes\vec u\in L^2((0,T), L^2(w\, dx))$\end{proposition}
  \begin{proof} Let $\phi_R=\theta(\frac x R) \frac 1 {(1+\vert x\vert)^2)^{N/2}}$ where $\theta\in\mathcal{D}$ is equal to $1$ on a neighborhood of $0$. We know that $\vec u$ is smooth, so that, for $0<t_0\leq t<T$, \begin{equation*}  \partial_t(\vert\vec u\vert^2)+ 2 \vert\vec\nabla\otimes\vec u\vert^2=\Delta(\vert\vec u\vert^2)- \Div((2p+\vert\vec u\vert^2)\vec u)\end{equation*}
and thus
 \begin{equation*}\begin{split}&\int \phi_R(x)\vert \vec u(t,x)\vert^2\, dx+2 \int_{t_0}^t \int \phi_R(x) \vert\vec\nabla\otimes\vec u(s,x)\vert^2\, dx\, ds \\=& \int \phi_R(x) \vert\vec u(t_0, x)\vert^2\, dx
 +\int_{t_0}^t \int  \Delta(\phi_R(x))  \vert \vec u(t,x)\vert^2\, dx\, ds\\& + \int_{t_0}^t \int (2p+\vert \vec u\vert^2) \vec u\cdot\vec\nabla (\phi_R(x))\, dx\, ds.
\end{split} \end{equation*}
  We have, for $\vert \alpha\vert\leq 2$, $\vert \partial^\alpha (\phi_R)\vert \leq C w$. On the other hand, we know that $\vec u\in L^\infty(L^2(w\, dx))$, that $\sqrt t u_iu_j\in L^\infty(L^2(w\, dx))$, and thus $\sqrt t(2p+\vert \vec u\vert^2)\in L^\infty(L^2(w\, dx))$ (as $w\in\mathcal{A}_2$ and $p=-\sum_{1\leq i\leq 3}\sum_{j=1}^3 \frac{\partial_i\partial_j}\Delta(u_iu_j)$), thus we get that
  \begin{equation*}\begin{split}&\int \phi_R(x)\vert \vec u(t,x)\vert^2\, dx+2 \int_{t_0}^t \int \phi_R(x) \vert\vec\nabla\otimes\vec u(s,x)\vert^2\, dx\, ds \\\leq &C \sup_{0<s<T} \int  \vert\vec u(s, x)\vert^2 \, w(x)\, dx
 +C\int_{0}^T \int    \vert \vec u(t,x)\vert^2 \, w(x)\, dx\, ds\\& + \int_{0}^T \int \sqrt s  \left\vert 2p+\vert \vec u\vert^2\right\vert  \, \vert \vec u\vert\, w(x)\, dx\, \frac{ds}{\sqrt s} <+\infty.
\end{split} \end{equation*}
We then let $R$ go to $+\infty$ and $t_0$ go to $0$.   \end{proof}

\section{Barker's stability theorem}

In this section, we extend a lemma of Barker on Leray weak solutions with initial values in $L^2\cap [L^2,\dot B^{-\delta}_{\infty,\infty}]_{\theta,\infty}$ (for some $\delta<1 $ and $\theta\in (0,1)$) to the case of some solutions 
 with initial values in  $L^2(w\, dx)\cap [L^2(w\, dx), H^{-\gamma}_r]_{\theta,\infty}$ where $w=\frac1{(1+\vert x\vert)^N}$ and $0\leq N\leq 2$, and $\gamma+\frac 3 r<1$.
 
 \begin{definition}$\ $\\ A weighted Leray weak solution for the Navier--Stokes equations with divergence-free  initial value $\vec u_0\in L^2(w\, dx)$, where $w=\frac1{(1+\vert x\vert)^N}$ and $0\leq N\leq 2$, is a divergence-free vector field $\vec u$ defined on $(0,T)\times\mathbb{R}^3$ such that
 \begin{itemize}
 \item[$\bullet$]  $\vec u\in L^\infty((0,T),L^2(w\, dx))$ and $\vec\nabla\otimes\vec u\in L^2((0,T), L^2(w\, dx))$
 \item[$\bullet$] there exists $p\in \mathcal{D}'((0,T)\times\mathbb{R}^3$ such that $$\partial_t\vec u=\Delta\vec u-\vec u\cdot\vec\nabla \vec u-\vec\nabla p$$
 \item[$\bullet$] $\lim_{t\rightarrow 0} \|\vec u(t,.)-\vec u_0\|_{L^2(w\, dx)}=0$
 \item[$\bullet$] $\vec u$ fulfills the weighted Leray inequality: for $0<t<T$,
 \begin{equation*}\begin{split}
 \int \vert \vec u(t,x)\vert^2\, w(x)\, dx &+2 \int_0^t \int \vert \vec \nabla\otimes\vec u(s,x)\vert^2\, w(x)\, dx\, ds
 \\ \leq   \int \vert \vec u_0(t,x)\vert^2&\, w(x)\, dx -2\sum_{i=1}^3 \int_0^t\int \partial_iw(s,x) \vec u(s,x)\cdot\partial_i\vec u(s,x)\, dx\, ds\\&+\int_0^t\int (\vert \vec u(s,x)\vert^2+2 p(s,x)) \vec u(s,x)\cdot\vec\nabla w(x)\, dx\, ds
 \end{split}\end{equation*}
\end{itemize} 
 \end{definition}
 
 The Navier--Stokes problem in $L^2(w\, dx)$ has recently been studied by Bradshaw, Kukavica  and Tsai \cite{BKT}, and Fern\'andez-Dalgo and Lemari\'e-Rieusset \cite{FLR20}. As $\vert\vec\nabla w\vert\leq N w$, we find that $\sqrt w \vec u\in L^2((0,T), H^1)$. In particular, we have $w u_i u_j\in L^{4}((0,T),L^{6/5})$. The pressure $p$ is determined by the equation $\Delta p=-\sum_{i=1}^3\sum_{j=1}^3 u_iu_j$ (see \cite{FLR21}) and, as $w^{6/5}\in\mathcal{A}_{6/5}$, we have $p\in L^4((0,T),L^{6/5}(w^{6/5}\, dx))$. As $\vert\vec\nabla w\vert\leq N w^{3/2}$, we see that the right-hand side of the weighted Leray inequality is well-defined.As in the case of Leray solutions, the strong continuity at $t=0$ of $t\in [0,T)\mapsto \vec u(t,.)\in L^2(w\, dx)$ (which is only weakly continuous for $t>0$) is a consequence of the weighted Leray inequality.
 
 \begin{theorem}\label{stb}$\ $\\ Let $\vec u_0$ be a divergence-free vector field such that $\vec u_0\in L^2(w\, dx)$, where $w=\frac1{(1+\vert x\vert)^N}$ and $0\leq N\leq 2$. Let $\vec u_1$,  $\vec u_2$ be  two weighted Leray weak solutions for the Navier--Stokes equations with  initial value $\vec u_0$.  If moreover $\vec u_0$ belongs to $[L^2(w\, dx), H^{-\gamma}_r]_{\theta,\infty}$ for some $\gamma>0$, $2<r<+\infty$ with $\gamma+\frac 3 r<1$ and $\theta\in (0,1)$, then  there exists $T_0>0$, $C\geq 0$  and $\eta>0$ such that, for $0\leq t\leq T_0$, 
 $$ \|\vec u_1(t,.)-\vec u_2(t,.)\|_{L^2(w\, dx)}\leq C t^\eta.$$
 \end{theorem}
 
 \begin{proof} This theorem was proved by Barker \cite{Bar18} in the case $N=0$. Our proof will follow the same lines as Barker's proof. 
 
 As $\vec u_0\in [L^2(w\, dx), H^{-\gamma}_r]_{\theta,\infty}$, for every $\epsilon\in (0,1)$ we may split $\vec u_0$ in $\vec u_0=\vec v_{0,\epsilon}+\vec w_{0,\epsilon}$ with $\|\vec v_{0,\epsilon}\|_{H^{-\gamma}_r}\leq C_1 \epsilon^{\theta-1}$ and $\|\vec w_{0,\epsilon}\|_{L^2(w\, dx)}\leq C_1 \epsilon^\theta$, where $C_1$ depends only on $\vec u_0$. As $\vec u_0=\mathbb{P} \vec u_0$ and as $\mathbb{P}$ is continuous on $H^{-\gamma}_r$ and on $L^2(w\, dx)$, we may assume (changing the value of the constant $C_1$) that $\vec v_{0,\epsilon}$ and $\vec w_{0,\epsilon}$ are divergence free. Let $\delta=\gamma+\frac 3 r<1$. Since $H^{-\gamma}_r\subset B^{-\delta}_{\infty,\infty}$, we have for $0<t\leq 1$, $\| e^{t\Delta}\vec v_{0,\epsilon}\|_\infty\leq C_2 t^{-\delta/2} \epsilon^{\theta-1}$. If $0<T_1<1$, we have $$\sup_{0<t<T_1} \sqrt t \|e^{t\Delta}\vec v_{0,\epsilon}\|_\infty\leq C_2 \epsilon^{\theta-1} T_1^{\frac {1-\delta}2}$$
 and
 $$\sup_{0<t<T_1, x\in\mathbb{R}^3} \sqrt{\frac 1{t^{3/2}} \int_0^t\int_{B(x,\sqrt t)} \vert e^{t\Delta}\vec v_{0,\epsilon}\vert^2\, dx}\leq C_3 \epsilon^{\theta-1} T_1^{\frac {1-\delta}2}
 $$ so that $\|e^{t\Delta}\vec v_{0,\epsilon}\|_{X_{T_1}} \leq (C_2+C_3) \epsilon^{\theta-1} T_1^{\frac {1-\delta}2}<\frac 1{8C_0}$ if $T_1<\min(1,C_4 \epsilon^{\frac 2{1-\delta}(1-\theta)})$.
 
 By (the proof of) Theorem \ref{solbesov}, we know that the Navier--Stokes equations with initial value $\vec v_{0,\epsilon}$ will have a solution $\vec v_\epsilon$ on $(0,T_1)$ such that $\|\vec v_\epsilon(t,.)\|_\infty\leq C_5  t^{-\delta/2} \epsilon^{\theta-1}$.   Moreover, by Proposition \ref{propweight}, $\vec v_\epsilon$ is a weighted Leray weak solution. 
 
 Let $\vec u$ be a weighted Leray solution on $(0,T)$  for the Navier--Stokes equations with  initial value $\vec u_0$.  We are going to compare $\vec u$ and $\vec v_\epsilon$. We know that $\vec v_\epsilon$ is smooth, so that $\partial_t(\vec u\cdot\vec v_\epsilon)=\vec u\cdot \partial_t\vec v_\epsilon+ \vec v_\epsilon\cdot \partial_t\vec u$. If $p_\epsilon$ is the pressure associated to $\vec v_\epsilon$, we have on $(0,T_2)$ where $T_2=\min(T,T_1)$
 \begin{equation*} \begin{split} \partial_t(\vec u\cdot\vec v_\epsilon)=& \vec u\cdot \Delta\vec v_\epsilon+ \vec v_\epsilon\cdot \Delta\vec u -\Div(p_\epsilon \vec u+p\vec v_\epsilon)  -\vec u\cdot(\vec v_\epsilon\cdot\vec \nabla\vec v_\epsilon)-\vec v_\epsilon\cdot(\vec u\cdot\vec\nabla \vec u)
\\ =& \vec u\cdot \Delta\vec v_\epsilon+ \vec v_\epsilon\cdot \Delta\vec u -\Div(p_\epsilon \vec u+p\vec v_\epsilon) \\& -(\vec u-\vec v_\epsilon)\cdot(\vec v_\epsilon\cdot\vec \nabla\vec v_\epsilon)-\vec v_\epsilon\cdot(\vec u\cdot\vec\nabla (\vec u-\vec v_\epsilon))-\Div(\frac{\vert\vec v_\epsilon\vert^2}2(\vec u+\vec v_\epsilon))
\\ =& \vec u\cdot \Delta\vec v_\epsilon+ \vec v_\epsilon\cdot \Delta\vec u-\vec v_\epsilon\cdot((\vec u-\vec v_\epsilon)\cdot\vec\nabla (\vec u-\vec v_\epsilon))
\\& -\Div(p_\epsilon \vec u+p\vec v_\epsilon + \frac{\vert\vec v_\epsilon\vert^2}2(\vec u+\vec v_\epsilon)+ (\vec v_\epsilon\cdot (\vec u-\vec v_\epsilon))\vec v_\epsilon) . 
\end{split} \end{equation*}
As $\vec v_\epsilon\in L^2((0,T_2), L^\infty)$, this can be integrated on $(0,t)\times\mathbb{R}^3$ against the measure $w(x)\, dx\, ds$ and gives

%%%%%

 \begin{equation*} \begin{split}  \int \vec u\cdot\vec v_\epsilon \, w(x)\, dx&-\int \vec u_0\cdot\vec v_{0,\epsilon}\, w(x)\, dx\\= -\int_0^t\int \sum_{i=1}^3& \partial_iw(x) (\vec u(s,x)\cdot \partial_i \vec v_\epsilon(s,x)+\vec v_\epsilon(s,x)\cdot \partial_i \vec u(s,x))\, dx\, ds\\& -2\int_0^t\int (\vec\nabla\otimes\vec u(s,x)\cdot\vec\nabla\otimes\vec v_\epsilon(s,x))\, w(x)\, dx\, ds
 \\  - \int_0^t\int \vec v_\epsilon(s,x)\cdot&((\vec u(s,x)-\vec v_\epsilon(s,x))\cdot\vec\nabla(\vec u(s,x)-\vec v_\epsilon(s,x)))\, w(x)\, dx\, ds
 \\& +\int_0^t\int p(s,x) \vec v_\epsilon(s,x)\cdot \vec \nabla w(x)+   p_\epsilon(s,x) \vec u(s,x)\cdot \vec \nabla w(x) \, dx\, ds
 \\+\int_0^t\int \frac{\vert \vec v_\epsilon(s,x)\vert^2}2 (\vec u(s,x)&-\vec v_\epsilon(s,x))\cdot\vec\nabla w(x) + (\vec v_\epsilon(s,x)\cdot\vec u(s,x))\vec v_\epsilon(s,x)\cdot \vec\nabla w(x)\, dx\, ds.
\end{split} \end{equation*}

  %%%%
  Together with
   \begin{equation*}\begin{split}
 \int \vert \vec u(t,x)\vert^2\, w(x)\, dx &+2 \int_0^t \int \vert \vec \nabla\otimes\vec u(s,x)\vert^2\, w(x)\, dx\, ds
 \\ \leq   \int \vert \vec u_0(t,x)\vert^2&\, w(x)\, dx -2\sum_{i=1}^3 \int_0^t\int \partial_iw(s,x) \vec u(s,x)\cdot\partial_i\vec u(s,x)\, dx\, ds\\&+\int_0^t\int (\vert \vec u(s,x)\vert^2+2 p(s,x)) \vec u(s,x)\cdot\vec\nabla w(x)\, dx\, ds
 \end{split}\end{equation*}
 and
  \begin{equation*}\begin{split}
 \int \vert \vec v_\epsilon(t,x)\vert^2\, w(x)\, dx &+2 \int_0^t \int \vert \vec \nabla\otimes\vec  v_\epsilon((s,x)\vert^2\, w(x)\, dx\, ds
 \\ =   \int \vert \vec  v_{0,\epsilon}(t,x)\vert^2&\, w(x)\, dx -2\sum_{i=1}^3 \int_0^t\int \partial_iw(s,x) \vec  v_\epsilon((s,x)\cdot\partial_i\vec  v_\epsilon((s,x)\, dx\, ds\\&+\int_0^t\int (\vert \vec  v_\epsilon((s,x)\vert^2+2 p_\epsilon(s,x)) \vec  v_\epsilon((s,x)\cdot\vec\nabla w(x)\, dx\, ds,
 \end{split}\end{equation*}
 this gives
  \begin{equation*}\begin{split}
 \int \vert \vec v_\epsilon(t,x)-\vec u(t,x)\vert^2\, w(x)\, dx &+2 \int_0^t \int \vert \vec \nabla\otimes(\vec  v_\epsilon-\vec u)\vert^2\, w(x)\, dx\, ds
 \\ \leq   \int \vert \vec  v_{0,\epsilon}-\vec u_0\vert^2&\, w(x)\, dx -2\sum_{i=1}^3 \int_0^t\int \partial_iw  (\vec  v_\epsilon-\vec u)\cdot\partial_i(\vec  v_\epsilon-\vec u)\, dx\, ds\\ + 2 \int_0^t \int (p_\epsilon-p) &(\vec v_\epsilon-\vec u)\cdot\vec\nabla w\, dx\, ds
    -2 \int_0^t\int \vec v_\epsilon\cdot((\vec u -\vec v_\epsilon)\cdot\vec\nabla(\vec u-\vec v_\epsilon))\, w\, dx\, ds 
  \\ & + \int_0^t \int \vert \vec v_\epsilon-\vec u\vert^2 \vec v_\epsilon\cdot  \vec\nabla w +(\vert \vec u\vert^2-\vert \vec v_\epsilon\vert^2)(\vec u-\vec v_\epsilon)\cdot\vec\nabla w\, dx\, ds.\end{split}\end{equation*}
  Thus, we have
  
  \begin{equation*}\begin{split}
 \int \vert \vec v_\epsilon(t,x)-\vec u(t,x)\vert^2\, w(x)\, dx &+2 \int_0^t \int \vert \vec \nabla\otimes(\vec  v_\epsilon-\vec u)\vert^2\, w(x)\, dx\, ds
 \\ \leq   \int \vert \vec  v_{0,\epsilon}-\vec u_0\vert^2&\, w(x)\, dx  + C_6  \int_0^t    \|\sqrt w(\vec u -\vec v_\epsilon)\|_2  \| \sqrt w \vec\nabla(\vec u-\vec v_\epsilon)\|_2\, ds\\ &+ C_6  \int_0^t \|(p-p_\epsilon)\, w \|_{L^{6/5}}  \|  \sqrt w(\vec v_\epsilon-\vec u)\|_6 \, ds
 \\&  + C_6  \int_0^t  \|\vec v_\epsilon\|_\infty \|\sqrt w(\vec u -\vec v_\epsilon)\|_2  \| \sqrt w \vec\nabla\otimes(\vec u-\vec v_\epsilon)\|_2\, ds 
  \\ & +C_6  \int_0^t \|\sqrt w (\vec u-\vec v_\epsilon)\|_3^2   (\|\sqrt w \vec u\|_3+\|\sqrt w \vec v_\epsilon\|_3)  \, ds.\end{split}\end{equation*}
  We have $$\|w(p-p_\epsilon)\|_{6/5}\leq C_7 \|w(\vec u\otimes\vec u-\vec v_\epsilon\otimes\vec v_\epsilon\|_{6/5}\leq C_7 \|\sqrt w(\vec u-\vec v_\epsilon)\|_2 (\|\sqrt w\vec u\|_3+\|\sqrt w\vec v_\epsilon\|_3)$$ 
  $$ \|\sqrt w(\vec u-\vec v_\epsilon)\|_3^2\leq \|\sqrt w(\vec u-\vec v_\epsilon)|_2\|\sqrt w(\vec u-\vec v_\epsilon)\|_6$$
  and $$ \|\sqrt w(\vec u-\vec v_\epsilon)\|_6\leq C_8 \left(\|\sqrt w(\vec u-\vec v_\epsilon)\|_2+\|\sqrt w\vec\nabla\otimes(\vec u-\vec v_\epsilon)\|_2\right),$$ so that
  \begin{equation*}\begin{split}
\|\sqrt w(\vec u(t,.)-\vec v_\epsilon(t,.))\|_2^2 &+2 \int_0^t \int \|\sqrt w\, \vec \nabla\otimes(\vec  v_\epsilon-\vec u)\|_2^2\, ds
 \\ \leq   \|\sqrt w(\vec  v_{0,\epsilon}-\vec u_0)\|_2^2&  + C_9  \int_0^t    \|\sqrt w(\vec u -\vec v_\epsilon)\|_2  \| \sqrt w \vec\nabla(\vec u-\vec v_\epsilon)\|_2\, ds\\ + C_9  \int_0^t   (\|\sqrt w(\vec u-\vec v_\epsilon)\|_2&+\|\sqrt w\vec\nabla\otimes(\vec u-\vec v_\epsilon)\|_2 )\|\sqrt w(\vec u-\vec v_\epsilon)\|_2 (\|\sqrt w\vec u\|_3+\|\sqrt w\vec v_\epsilon\|_3) \, ds
 \\&  + C_9  \int_0^t  \|\vec v_\epsilon\|_\infty \|\sqrt w(\vec u -\vec v_\epsilon)\|_2  \| \sqrt w \vec\nabla\otimes(\vec u-\vec v_\epsilon)\|_2\, ds  
 \\ \leq   \|\sqrt w(\vec  v_{0,\epsilon}-\vec u_0)\|_2^2& + \int_0^t   \|\sqrt w\, \vec \nabla\otimes(\vec  v_\epsilon-\vec u)\|_2^2\, ds
\\&  + C_{10}  \int_0^t    \|\sqrt w(\vec u -\vec v_\epsilon)\|_2^2 (1+\|\sqrt w\vec u\|_3^2+\|\sqrt w\vec v_\epsilon\|_3^2+ \|\vec v_\epsilon\|_\infty^2 )   \, ds.  \end{split}\end{equation*} By Gonwall's lemma, we find that, for $0<t<T_2$, we have
$$ \|\sqrt w(\vec u(t,.)-\vec v_\epsilon(t,.))\|_2^2 \leq \|\sqrt w(\vec  v_{0,\epsilon}-\vec u_0)\|_2^2 e^{\int_0^{T_2} C_{10}( 1+\|\sqrt w\vec u\|_3^2+\|\sqrt w\vec v_\epsilon\|_3^2+ \|\vec v_\epsilon\|_\infty^2)\, ds}.$$
We know that $T_2\leq T$, $\int_0^{T_2} \|\sqrt w\vec u\|_3^2 \, ds \leq \int_0^T \|\sqrt w\vec u\|_3^2\, ds<+\infty$, and, by Propositions \ref{propweight} and \ref{propenerg},  $\int_0^{T_2} \|\sqrt w\vec v_\epsilon\|_3^2 \, ds \leq \int_0^{T_1} \|\sqrt w\vec v_\epsilon\|_3^2\, ds\leq C_{11} \|\vec u_0\|_{L^2(w\, dx)}^2$.  Finally, we have
$$ \int_0^{T_2} \|\vec v_\epsilon\|_\infty^2)\, ds \leq C_{12} \int_0^{T_1} t^{-\delta} \|\vec v_{0,\epsilon}\|_{B^{-\delta}_{\infty,\infty}}^2\, dt \leq C_{13} T_1^{1-\delta}\epsilon^{2(\theta-1)}\leq C_{14}.$$
Thus, we have
$$\|\sqrt w(\vec u(t,.)-\vec v_\epsilon(t,.))\|_2^2 \leq C_{15} \epsilon^{2\theta}.$$ $C_{15}$ depends only on $\vec u$ and $\vec u_0$.

We may now estimate $\|\vec u_1(t,.)-\vec u_2(t,.)\|_{L^2(w\, dx)}$ for two weighted Leray weak solutions defined on $(0,T)$. If $t\in (0,T)$, we define $\epsilon= (\frac 4 {C_4} t)^{\frac{1-\delta}{2(1-\theta)}}$ and $T_3=\frac 1 2 C_4 \epsilon^{\frac 2{1-\delta}(1-\theta)}=2t$. If $t$ is small enough, we have $0<\epsilon<1$ and $T_3<\min(1,T)$. Thus, we know that, for a constant $C$ that depends only on $\vec u_1$, $\vec u_2$ and $\vec u_0$, 
\begin{equation*}\begin{split} \|\vec u_1(t,.)-\vec u_2(t,.)\|_{L^2(w\, dx)} \leq& \|\vec u_1(t,.)-\vec v_\epsilon(t,.)\|_{L^2(w\, dx)} +\|\vec v_\epsilon(t,.)-\vec u_2(t,.)\|_{L^2(w\, dx)} 
\\ \leq & C \epsilon^\theta= C (\frac 4 {C_4} t)^{\theta\frac{1-\delta}{2(1-\theta)}}
\end{split}\end{equation*}
The theorem is proved.   \end{proof}
\section{Weak-strong uniqueness}

We may now prove  Theorem \ref{theoweight}. 

\begin{proof}Recall that we consider two  solutions  $\vec u$, $\vec v$  of the Navier--Stokes equations on $(0,T)$   with the same initial value $\vec u_0 $ such that:
\begin{itemize}
\item[$\bullet$]  $\vec u_0$ be a divergence-free vector field with $\vec u_0\in L^2\cap   {\rm bmo}^{-1}_0$
\item[$\bullet$] $\|e^{t\Delta}\vec u_0\|_{X_T}< \frac 1 {4C_0} $ 
\item[$\bullet$]  $\vec u$ is  the mild solution  of the Navier--Stokes equations with initial value $\vec u_0$ such that   $\|\vec u\|_{X_T}\leq \frac 1 {2C_0} $ 
\item[$\bullet$] for some  $N\geq 0$ , $2<q<+\infty$  { and } $0\leq  s<1-\frac 2 q$,   $$ \sup_{0<t<T} t^{s/2} \|\vec u\|_{L^q(\frac 1{(1+\vert x\vert)^N}\, dx)}<+\infty$$
\item[$\bullet$]  $\vec v$ is a suitable weak Leray solution of the Navier--Stokes equations.
\end{itemize}

We know, by Propositions \ref{propweight} and \ref{propenerg}, that the mild solution $\vec u$ is a  suitable weak Leray solution. In particular, we have $\sup_{0<t<T} \|\vec u(t,.)\|_2<+\infty$, while $\sup_{0<t<T} t^{1/2} \|\vec u(t,.)\|_\infty\leq  \|\vec u\|_{X_T}<+\infty$. Thus, $$\sup_{0<s<T} t^{\frac 1 2-\frac 1{q}}\|\vec u\|_q<+\infty.$$  If $0\leq \alpha\leq 1$, we find that 
$$ \sup_{0<t<T} (\sqrt t)^{(1-\alpha)(1-\frac 2 q)+\alpha s} \|\vec u\|_{L^q(\frac 1{(1+\vert x\vert)^{\alpha N}}\, dx)} <+\infty.$$
By Theorem \ref{solbesov}, we find that $$\vec u_0\in B^{-s_\alpha}_{L^q(\frac 1{(1+\vert x\vert)^{\alpha N}}\, dx,\infty}\text{ with } s_\alpha =(1-\alpha)(1-\frac 2 q)+\alpha s.$$
For $0<\alpha<\min(1,\frac 4{Nq})$, $0<s_\alpha<1-\frac 2 q$ and $\alpha N<\frac 4 q$, so that we may apply Corollary \ref{embed}. 

  The next step is to check that $\vec u$ and $\vec v$, that are suitable Leray weak solutions, are weighted Leray weak solutions, for the weight $w(x)=\frac 1{(1+\vert x\vert)^2}$.  It means that we must check that $\vec v$ (and $\vec u$) fulfills the weighted Leray energy inequality. We consider a non-negative function $\theta\in \mathcal{D}(\mathbb{R}^3)$ equal to $1$ on a neighborhood of $0$ and $0$ for $\vert x\vert\geq 1$ and a function $\alpha$ smooth  on $\mathbb{R}$, such that $0\leq \alpha\leq 1$, with $\alpha(t)$ equal to $0$ on $(\infty,0)$ and $1$ on $(1,+\infty)$. For   $0<t_0<t_1<T$, $R>0$ and $0<\epsilon< \min(t_1-t_0, T-t_1)$, we define the test function 
  $$\varphi_{t_0,t_1,\epsilon,R}(t,x)=\alpha(\frac{t-t_0}\epsilon) (1-\alpha(\frac{t-t_1}\epsilon))  \frac1 {(1+\sqrt{\frac 1{R^2}+x^2})^2} \theta(\frac x R)=\alpha_{t_0,t_1,\epsilon}(t) \theta_R(x)$$ which is non-negative and supported in $[t_0,t_1+\epsilon ]\times \overline{B(0,R)}$. We have, by suitability of $\vec v$, if $q$ is the pressure associated to the solution $\vec v$, 
  $$ \iint \varphi_{t_0,t_1,\epsilon,R} \left(\partial_t(\vert\vec v\vert^2)+ 2 \vert\vec\nabla\otimes\vec v\vert^2-\Delta(\vert\vec v\vert^2)+ \Div((2q+\vert\vec v\vert^2)\vec  v)\right)
  \, dx\, dt\leq 0.$$
  As , for $R\geq 1$, $\vert \theta_R\vert \leq C w$ and $\vert\vec\nabla \theta_R\vert\leq C w^{3/2}$, dominated convergence when $R$ goes to $+\infty$ gives us
\begin{equation*}\begin{split}& \iint  (\frac 1 \epsilon \alpha'(\frac{t-t_1}\epsilon)-\frac 1 \epsilon \alpha'(\frac{t-t_0}\epsilon)) \vert \vec v\vert^2\, w(x)\, dx\, dt + 2\iint \alpha_{t_0,t_1,\epsilon} \vert\vec\nabla\otimes\vec v\vert^2\, w(x)\, dx\, dt 
\\ \leq& -2 \sum_{i=1}^3 \iint  \alpha_{t_0,t_1,\epsilon}  \partial_iw (\vec v\cdot \partial_i\vec v) \, dx\, dt 
+\iint  \alpha_{t_0,t_1,\epsilon}  (\vert\vec v\vert^2+2q) \vec v\cdot \vec\nabla w\, dx\, dt \end{split}\end{equation*}  If $\epsilon$ goes to $0$, we get
\begin{equation*}\begin{split}& \limsup_{\epsilon\rightarrow 0}  \int  (\frac 1 \epsilon \alpha'(\frac{s-t_1}\epsilon)-\frac 1 \epsilon \alpha'(\frac{s-t_0}\epsilon))( \int \vert \vec v(s,x)\vert^2\, w(x)\, dx)\, ds  \\&+ 2\int_{t_0}^{t_1} \int\vert\vec\nabla\otimes\vec v\vert^2\, w(x)\, dx\, ds 
\\ \leq& -2 \sum_{i=1}^3 \int_{t_0}^{t_1} \int  \partial_iw (\vec v\cdot \partial_i\vec v) \, dx\, ds.
+\int_{t_0}^{t_1} \int (\vert\vec v\vert^2+2q) \vec v\cdot \vec\nabla w\, dx\, ds \end{split}\end{equation*}  For almost every $t_0$, $t_1$, we have that $t_0$ and $t_1$ are Lebesgue points of the map $s\mapsto \int \vert \vec v(s,x)\vert^2\, w(x))\, dx$, so that
\begin{equation*}\begin{split} \lim_{\epsilon\rightarrow 0}  \int  (\frac 1 \epsilon \alpha'(\frac{s-t_1}\epsilon)-\frac 1 \epsilon \alpha'(\frac{s-t_0}\epsilon))( \int \vert \vec v(s,x)\vert^2\, w(x)\, dx)\, ds \\=  \int \vert \vec v(t_1,x)\vert^2\, w(x)\, dx- \int \vert \vec v(t_0,x)\vert^2\, w(x)\, dx.\end{split}\end{equation*} 
 If $t_0$ goes to $0$ and $t_1$ goes to $t$, we have $\|\vec v(t_0,.)-\vec u_0\|_{L^2(w\, dx)}\leq \| \vec v(t_0,.)-\vec u_0\|_2\rightarrow 0$, so that $$\lim_{t_0\rightarrow 0} \int \vert \vec v(t_0,x)\vert^2\, w(x)\, dx=\int \vert \vec u_0(x)\vert^2\, w(x)\, dx$$ while $\vec v(t_1,.)$ is weakly convergent to $\vec v(t,.)$ so that
$$ \int \vert \vec v(t,x)\vert^2\, w(x)\, dx\leq\liminf_{t_1\rightarrow t}\int \vert \vec v(t_1,x)\vert^2\, w(x)\, dx.$$ Thus, we get the  weighted Leray energy inequality.
  
By Theorem \ref{stb}, we then know that  there exists $T_0>0$, $C\geq 0$  and $\eta>0$ such that, for $0\leq t\leq T_0$, 
 $$ \|\vec u(t,.)-\vec v(t,.)\|_{L^2(w\, dx)}\leq C t^\eta.$$
Moreover, we can do the same computations as in the proof of Theorem \ref{stb} in order to estimate $\partial_t(\vec u.\vec v)$ (since $\vec u$ is smooth)  and write, if $p$ is the pressure associated to $\vec u$ and $q$ the pressure associated to $\vec v$, 
 \begin{equation*} \begin{split} \partial_t(\vec u\cdot\vec v)=& \vec u\cdot \Delta\vec v+ \vec v\cdot \Delta\vec u-\vec u\cdot((\vec u-\vec v)\cdot\vec\nabla (\vec u-\vec v))
\\& -\Div(q \vec u+p\vec v + \frac{\vert\vec u\vert^2}2(\vec u+\vec v)+ (\vec u\cdot (\vec v-\vec u))\vec u) . 
\end{split} \end{equation*}

As $\vec u\in L^2((\epsilon,T), L^\infty)$ for every $\epsilon>0$, this can be integrated on $(\epsilon,t)\times\mathbb{R}^3$ against the measure $w(x)\, dx\, ds$ and gives

%%%%%

 \begin{equation*} \begin{split}  \int \vec u(t,x)\cdot&\vec v(t,x)  \, w(x)\, dx-\int \vec u(\epsilon,x)\cdot\vec v(\epsilon,x))\, w(x)\, dx\\=&-\int_\epsilon^t\int \sum_{i=1}^3 \partial_iw (\vec u\cdot \partial_i \vec v+\vec v\cdot \partial_i \vec u)\, dx\, ds\\& -2\int_\epsilon^t\int (\vec\nabla\otimes\vec u \cdot\vec\nabla\otimes\vec v)\, w(x)\, dx\, ds
 \\  - \int_\epsilon^t\int \vec u\cdot&((\vec u-\vec v)\cdot\vec\nabla(\vec u -\vec v))\, w(x)\, dx\, ds
 \\& +\int_\epsilon^t\int p  \vec v\cdot \vec \nabla w +   q \vec u\cdot \vec \nabla w \, dx\, ds
 \\+\int_\epsilon^t\int \frac{\vert \vec u\vert^2}2 (\vec v&-\vec u)\cdot\vec\nabla w + (\vec v\cdot\vec u)\vec u\cdot \vec\nabla w(x)\, dx\, ds.
\end{split} \end{equation*}
As $\vec u(\epsilon,.)$ and $\vec v(\epsilon,.)$ are strongly convergent to $\vec u_0$ in $L^2(w\, dx)$, we find
 \begin{equation*} \begin{split}  \int \vec u(t,x)\cdot&\vec v(t,x)  \, w(x)\, dx-\int \vec u_0\cdot\vec u_0\, w(x)\, dx\\=&-\int_0^t\int \sum_{i=1}^3 \partial_iw (\vec u\cdot \partial_i \vec v+\vec v\cdot \partial_i \vec u)\, dx\, ds\\& -2\int_0^t\int (\vec\nabla\otimes\vec u \cdot\vec\nabla\otimes\vec v)\, w \, dx\, ds
 \\  -\lim_{\epsilon\rightarrow 0} \int_\epsilon^t\int \vec u\cdot&((\vec u-\vec v)\cdot\vec\nabla(\vec u -\vec v))\, w\, dx\, ds
 \\& +\int_0^t\int p  \vec v\cdot \vec \nabla w +   q \vec u\cdot \vec \nabla w \, dx\, ds
 \\+\int_0^t\int \frac{\vert \vec u\vert^2}2 (\vec v&-\vec u)\cdot\vec\nabla w + (\vec v\cdot\vec u)\vec u\cdot \vec\nabla w\, dx\, ds.
\end{split} \end{equation*}
We have 
$$\lim_{\epsilon\rightarrow 0} \int_\epsilon^t\int \vec u\cdot((\vec u-\vec v)\cdot\vec\nabla(\vec u -\vec v))\, w\, dx\, ds=\int_0^t \int  s^\eta\vec u\cdot s^{-\eta}((\vec u-\vec v)\cdot\vec\nabla(\vec u -\vec v))\, w\, dx\, ds$$ as $s^\eta \vec u\in L^2 L^\infty$, $s^{-\eta}(\vec u-\vec v)\in L^\infty(L^2(w\, dx))$ and $\vec\nabla\otimes(\vec u-\vec v)\in L^2(L^2(w\, dx))$.
Using now the weighted Leray inequalities on $\vec v$ and on $\vec u$, we get   \begin{equation*}\begin{split}
 \int \vert \vec v(t,x)-\vec u(t,x)\vert^2\, w(x)\, dx &+2 \int_0^t \int \vert \vec \nabla\otimes(\vec  v-\vec u)\vert^2\, w\, dx\, ds
 \\ \leq   & -2\sum_{i=1}^3 \int_0^t\int \partial_iw  (\vec  v-\vec u)\cdot\partial_i(\vec  v-\vec u)\, dx\, ds\\ + 2 \int_0^t \int (q-p) &(\vec v-\vec u)\cdot\vec\nabla w\, dx\, ds
    -2 \int_0^t\int \vec u\cdot((\vec u -\vec v)\cdot\vec\nabla(\vec u-\vec v))\, w\, dx\, ds 
  \\ & + \int_0^t \int \vert \vec v-\vec u\vert^2 \vec u\cdot  \vec\nabla w +(\vert \vec u\vert^2-\vert \vec v\vert^2)(\vec u-\vec v)\cdot\vec\nabla w\, dx\, ds,\end{split}\end{equation*}
and thus  \begin{equation*}\begin{split}
 \int \vert \vec v(t,x)-\vec u(t,x)\vert^2\, w(x)\, dx &+2 \int_0^t \int \vert \vec \nabla\otimes(\vec  v-\vec u)\vert^2\, w\, dx\, ds
 \\ \leq &    C   \int_0^t    \|\sqrt w(\vec u -\vec v)\|_2  \| \sqrt w \vec\nabla(\vec u-\vec v)\|_2\, ds\\ &+ C  \int_0^t \|(p-q)\, w \|_{L^{6/5}}  \|  \sqrt w(\vec v-\vec u)\|_6 \, ds
 \\&  + C   \int_0^t  \|\vec u\|_\infty \|\sqrt w(\vec u -\vec v)\|_2  \| \sqrt w \vec\nabla\otimes(\vec u-\vec v)\|_2\, ds 
  \\ & +C  \int_0^t \|\sqrt w (\vec u-\vec v)\|_3^2   (\|\sqrt w \vec u\|_3+\|\sqrt w \vec v\|_3)  \, ds.\end{split}\end{equation*}
At this point, we get

  \begin{equation*}\begin{split}
\|\sqrt w(\vec u(t,.)-\vec v(t,.))\|_2^2 &+2 \int_0^t \int \|\sqrt w\, \vec \nabla\otimes(\vec  v-\vec u)\|_2^2\, ds
 \\ \leq    &  \int_0^t   \|\sqrt w\, \vec \nabla\otimes(\vec  v-\vec u)\|_2^2\, ds
\\&  + C  \int_0^t    \|\sqrt w(\vec u -\vec v)\|_2^2 (1+\|\sqrt w\vec u\|_3^2+\|\sqrt w\vec v\|_3^2+ \|\vec u\|_\infty^2 )   \, ds.  \end{split}\end{equation*} 
Let $$A(t)= t^{-2\eta} \|\sqrt w(\vec u(t,.)-\vec v(t,.))\|_2^2 $$ and $$B(t)=1+\|\sqrt w\vec u\|_3^2+\|\sqrt w\vec v\|_3^2.$$ We have
$$ A(t)\leq C \int_0^t A(s) B(s)\, ds+ Ct^{-2\eta} \int_0^t A(s) s^{2\eta}  \|\vec u\|_\infty^2    \, ds.  $$
Thus, for $0<t<\tau<T$,
$$ A(t)\leq C  \sup_{0<s<\tau} A(s) (\int_0^\tau B(s)\, ds+ \frac 1{2\eta} \sup_{0<s<\tau}  s \|\vec u(s,.)\|_\infty^2 ).$$ For $\tau$ small enough, we have
$$   C    (\int_0^\tau B(s)\, ds+ \frac 1{2\eta} \sup_{0<s<\tau}  s \|\vec u(s,.)\|_\infty^2 )<1$$ and thus $\sup_{0<t<\tau} A(t)=0$. We conclude that $\vec u=\vec v$ on $[0,\tau]$. Since $\vec u$ is bounded on $[\tau, T]$, then uniqueness is easily extended to the whole interval $[0,T]$.
 \end{proof}

\section{Further comments on Barker's conjecture}

In his paper \cite{Bar18},  Barker   raised the following question :

\begin{question} If $\vec u_0$ belongs to $L^2\cap {\rm bmo}_0^{-1}$, does there exists a positive time $T$ such that every weak Leray solution of the Cauchy problem for the Navier--Stokes equations with $\vec u_0$ as initial value coincide with the mild solution in $X_T$ ?
\end{question}

This can be seen as the endpoint case of the Prodi--Serrin weak-strong uniqueness criterion, as the assumption  of Prodi--Serrin's criterion, i.e. existence of a solution $\vec u$ such that  $$\vec u\in   L^p_t L^q_x\text{ with } \frac 2 p+\frac 3 q\leq1 \text{ and } 2\leq p\leq +\infty,$$  is equivalent, if $2<p<+\infty$, to the fact that $\vec u_0$ belongs to $  B^{-1+\frac 3 q}_{q,p}\subset  {\rm bmo}_0^{-1}$. Existence of a mild solution when $\vec u_0$ belongs to $  B^{-1+\frac 3 q}_{q,p}$ goes back to the paper of Fabes, Jones and Rivi\`ere \cite{FJR72}. Existence of mild solutions has been extended by Cannone \cite{Can95} to the case of $  B^{-1+\frac 3 q}_{q,\infty}\cap   {\rm bmo}_0^{-1}$, and Koch and Tataru's theorem \cite{KoT01} can be seen as the endpoint case of the theory for existence of mild solutions.

Barker \cite{Bar18} extended weak-strong uniqueness to the case  $B^{-1+\frac 3 q}_{q,\infty}\cap   {\rm bmo}_0^{-1}$, and could even relax the regularity exponent and consider the case  $B^{-s}_{q,\infty}\cap   {\rm bmo}_0^{-1}$ with $s<1-\frac 2 q$. We have shown that the integrability could even be relaxed into 
 $B^{-1+\frac 3 q}_{L^q(\frac 1{(1+\vert x\vert)^N}\, dx),\infty}\cap   {\rm bmo}_0^{-1}$ with $N\geq 0$ and $s<1-\frac 2 q$. But under the sole assumption $\vec u_0\in L^2\cap {\rm bmo}^{-1}_0$, weak-strong uniqueness remains an open question.
 
 An alternative way to study the problem is to impose restrictions on the class of solutions, beyond the Leray energy inequality or the local Leray energy inequality. One may for instance consider an approximation process that provides weak Leray solutions when $\vec u_0\in L^2$ and consider whether the solutions provided by this process coincide with the mild solution when moreover $\vec u_0\in  {\rm bmo}^{-1}_0$. There are many processes that pave the way to Leray solutions (and in most cases to suitable weak Leray solutions); in \cite{Lem16}, we described fourteen different processes (including $\alpha$-models, frequency cut-off, damping, artificial viscosity, hyperviscosity,\dots). 
 
 The scheme is always the same. One approximates the Navier--Stokes equations (NS) by equations (NS$_\alpha$) depending on a small parameter $\alpha\in (0,1)$. Equations (NS$_\alpha$)  with initial value $\vec u_0\in L^2$ have  a unique solution $\vec u_\alpha$. One then establishes an energy (in)equality that allows to control $\vec u_\alpha$ uniformly on $L^\infty((0,T),L^2)\cap L^2((0,T),H^1)$. Moreover, one proves that $\partial_t\vec u_\alpha$ is controlled uniformly in $L^{6/5}((0,T),H^{-3})$.  By the Aubin--Lions theorem, there exists a sequence $\alpha_k\rightarrow 0$ such that $\vec u_{\alpha_k}$ is weakly convergent in $L^2((0,T),H^1)$ and strongly convergent in $(L^2((0,T)\times\mathbb{R}^3))_{\rm loc}$ to a limit $\vec v$. One then checks that $\vec v$ is a weak Leray solution of the Navier--Stokes equations with initial value $\vec u_0$.
 
 Some of those processes behave well for initial values $\vec u_0\in  {\rm bmo}^{-1}_0$, others don't seem to be well adapted to such initial values. More precisely, if one can prove that, when $\vec u_0$ belongs to $L^2\cap {\rm bmo}^{-1}_0$, there exists a time $T_0$ such that the solutions $\vec u_\alpha$ remain small in $X_{T_0}$ ($\|e^{t\Delta}\vec u_0\|_{X_{T_0}}<\eta<\frac 1{4_0}$ and $\sup_{\alpha\in (0,1)} \|\vec u_\alpha\|_{X_{T_0}}\leq 2\eta\leq   \frac 1 {2C_0}$), then the weak limit $\vec v$ will still remain controlled in $X_{T_0}$. But there is only one weak solution $\vec u$ in  $X_{T_0}$ such that $ \|\vec u\|_{X_{T_0}}\leq  \frac 1 {2C_0}$. Thus, the process cannot create a Leray solution that would escape the weak--strong uniqueness.
 
 Such processes can be found in processes that mimick Leray's mollification. 
Mollicication has been introduced by Leray \cite{Ler34} in his seminal paper on weak solutions for the Navier--Stokes equations. The approximated problem he considered  is the following one: solve
 $$\partial_t \vec u_\alpha+(\varphi_\alpha*\vec u_\alpha).\vec\nabla \vec u_\alpha=\Delta\vec u_\alpha -\vec\nabla p_\alpha$$ with $\text{ div }\vec u_\alpha=0$ and $\vec u_\alpha(0,.)=\vec u_0$. Here, $\varphi\in \mathcal{D}$, $\varphi\geq 0$, $\int\varphi\, dx=1$ and $\varphi_\alpha(x)=\frac 1{\alpha^3} \varphi(\frac x \alpha)$.
 Solving the mollified problem amounts to solve   the following integro-differential problem : 
\[ \vec  v=e^{t\Delta}\vec u_0-B(\varphi_\alpha*\vec v,\vec v)(t,x)\] where
\begin{equation*} B(\vec v,\vec w)=\int_0^t e^{(t-s)\Delta} \mathbb{P}\Div (\vec  v\otimes\vec w) \, ds.\end{equation*}   
Since $\|\varphi_\alpha*\vec v(t,.)\|_\infty\leq \|\vec v(t,.)\|_\infty$ and
\begin{equation*}\begin{split} (\int_0^t \int_{B(x_0,\sqrt t)} &\vert \varphi_\alpha*\vec v(s,.)(y)\vert^2\, dy\, ds)^{1/2}
\\=&  (\int_0^t \int_{B(x_0,\sqrt t)} \vert  \int \varphi_\alpha(z) \vec v(s,y-z)\, dz\vert^2\, dy\, ds)^{1/2}
\\\leq&  (\int_0^t \int_{B(x_0,\sqrt t)}\int       \varphi_\alpha(z)  \vert \vec v(s,y-z)\vert^2 \, dz  \, dy\, ds)^{1/2}
\\=&  (  \int \varphi_\alpha(z)  \left( \int_0^t \int_{B(x_0+z,\sqrt t)} \vert\vec v(s,y) \vert^2\, dy\, ds\right)\, dz)^{1/2},
\end{split}\end{equation*}
 we find that $\|\varphi_\alpha*\vec v\|_{X_T}\leq \|\vec v\|_{X_T}$.  Thus, the theorem of Koch and Tataru (Theorem \ref{theoKT} and Corollary \ref{kocht}) still applies :
 \begin{itemize}
 \item[$\bullet$] for every $\alpha>0$ and every $T>0$, we have
\[ \|B(\varphi_\alpha\vec v,\vec w)\|_{X_T}\leq C_0 \|\vec  v\|_{X_T}\|\vec w\|_{X_T}.\]
\item [$\bullet$] If $\|e^{t\Delta}\vec u_0\|_{X_T}<\frac 1{4C_0}$, then the mollified Navier--Stokes equations have a solution on $(0,T)$ such that $\|\vec u_\alpha\|_{X_T}\leq 2\| e^{t\Delta}\vec u_0\|_{X_T}$. \end{itemize}

Now, we may consider various other approximations of the Navier--Stokes equations of the form
\begin{equation}\label{eqalpha}\vec  v=e^{t\Delta}\vec u_0-\sum_{i=1}^N \varphi_{i,\alpha}*B_i(\psi_{i,\alpha}*\vec v,\chi_{i,\alpha}*\vec v)(t,x)  \end{equation} where 
\begin{itemize}
\item[$\bullet$] $\varphi_i$, $\psi_i$, $\chi_i$  are either the Dirac mass or functions in $L^1$
\item[$\bullet$]  $f_\alpha(x)=\frac 1{\alpha^3} f(\frac x \alpha)$ for $f\in \{\varphi_i, \psi_i,\chi_i, i=1,\dots N\}$
\item[$\bullet$] $B_i(\vec v,\vec w)=\int_0^t e^{(t-s)\Delta} \sigma_i(D)(\vec v\otimes\vec w)\, ds$ where $\sigma_i$ is given convolutions with smooth Fourier multipliers homogeneous of degree $1$: if $\vec z=\sigma_i(D)(\vec v\otimes\vec w)$, $z_k=\sum_{p,q\leq 3} K_{i,k,p,q}*(v_pw_q)$ where the Fourier transform of $K_{i,k,p,q}$ is  and homogenous of degree $1$ and is smooth on $\mathbb{R}^3$.
\end{itemize}
The proof of the  Koch and Tataru theorem asserts that operators of the form $B(\vec v,\vec w)=\int_0^t e^{(t-s)\Delta} \sigma(D)(\vec v\otimes\vec w)\, ds$ are bounded on $X_T$.

 Writing $\|\delta\|_1=1$, we have  
\begin{equation*}\begin{split} \|\sum_{i=1}^N \varphi_{i,\alpha}*B(\psi_{i,\alpha}*\vec v,\chi_{i,\alpha}*\vec v)(t,x)\|_{X_T} \\ \leq  (\sum_{i=1}^N \|B_i\|_{\rm op}\|\varphi_i\|_1\|\psi_i\|_1\|\chi_i\|_1) \|\vec  v\|_{X_T}\|\vec w\|_{X_T}
\\ =    C_1   \|\vec  v\|_{X_T}\|\vec w\|_{X_T} \end{split}\end{equation*}
 If $\|e^{t\Delta}\vec u_0\|_{X_T}<\frac 1{4C_1}$, then the modified equations (\ref{eqalpha})  have a solution on $(0,T)$ such that $\|\vec u_\alpha\|_{X_T}\leq 2\| e^{t\Delta}\vec u_0\|_{X_T}$. 
 
 Remark that the equations (\ref{eqalpha}) can be written as well
 $$ \partial_t\vec v= \Delta\vec v-\sum_{i=1}^N \varphi_{i,\alpha}*\sigma_i(D)((\psi_{i,\alpha}*\vec v)\otimes(\chi_{i,\alpha}*\vec v))$$ with initial value $\vec v(0,.)=\vec u_0$. 
Among example of such approximations, we have the various $\alpha$-models studied by Holm and Titi: 
\begin{itemize}
\item[$\bullet$] 
 The Leray-$\alpha$ model.\\
  The Leray--$\alpha$ model   has been discussed  in 2005 
       by Cheskidov, Holm,  Olson and Titi    \cite{CHOT05}.  The approximated problem is the following one: solve
 $$\partial_t \vec u_\alpha+((\Id-\alpha^2\Delta)^{-1}\vec u_\alpha).\vec\nabla \vec u_\alpha=\Delta\vec u_\alpha -\vec\nabla p_\alpha$$ with $\text{ div }\vec u_\alpha=0$ and $\vec u_\alpha(0,.)=\vec u_0$.  
  This is equivalent to write
 $$\partial_t \vec u_\alpha =\Delta\vec u_\alpha - \mathbb{P}\Div(((\Id-\alpha^2\Delta)^{-1}\vec u_\alpha)\otimes\vec u_\alpha). $$ 
 
\item[$\bullet$] 
 The Navier--Stokes-$\alpha$ model.\\
 The mathematical study of the Navier--Stokes-$\alpha$ model has been done  by Foias, Holm and Titi   in 2002\cite{FHT02}.  The approximated problem is the following one: solve
 $$\partial_t \vec u_\alpha+((\Id-\alpha^2\Delta)^{-1}\vec u_\alpha).\vec\nabla \vec u_\alpha=\Delta\vec u_\alpha-  \sum_{k=1}^3 u_{\alpha,k}\vec\nabla (\Id-\alpha^2\Delta)^{-1} u_{\alpha,k} -\vec\nabla p_\alpha$$ with $\text{ div }\vec u_\alpha=0$ and $\vec u_\alpha(0,.)=\vec u_0$.  We can rewrite the equation  as
\begin{equation*}\begin{split}  \partial_t \vec u_\alpha+((\Id-\alpha^2\Delta)^{-1}\vec u_\alpha).\vec\nabla \vec u_\alpha=&\Delta\vec u_\alpha-  \sum_{k=1}^3 ( \alpha^2 \Delta (\Id-\alpha^2\Delta)^{-1} u_{\alpha,k})\vec\nabla (\Id-\alpha^2\Delta)^{-1} u_{\alpha,k}\\& -\vec\nabla ( p_\alpha+\frac {\vert ((\Id-\alpha^2\Delta)^{-1}\vec u_\alpha\vert^2}2) \end{split}\end{equation*} 
  This is equivalent to write
\begin{equation*}\begin{split}  \partial_t \vec u_\alpha =\Delta\vec u_\alpha &- \mathbb{P}\Div(((\Id-\alpha^2\Delta)^{-1}\vec u_\alpha)\otimes\vec u_\alpha)\\&- \sum_{j=1}^3\sum_{k=1}^3  \mathbb{P}\partial_j ((\alpha\partial_j(\Id-\alpha^2\Delta)^{-1} u_{\alpha,k}) (\alpha\vec\nabla ( \Id-\alpha^2\Delta)^{-1} u_{\alpha,k} )).  \end{split}\end{equation*}

\item[$\bullet$] 
 The Clark-$\alpha$ model.\\
 The Clark-$\alpha$ model has been discussed in 2005 by Cao, Holm and Titi \cite{CHT05}. The approximated problem is the following one: solve
\begin{equation*}\begin{split} \partial_t\vec u_\alpha+ (\Id-\alpha^2\Delta)^{-1}\vec u_\alpha.\vec\nabla\, \vec u_\alpha=&   \Delta\vec u_\alpha +((\Id-\alpha^2\Delta)^{-1}\vec u_\alpha-\vec u_\alpha)\cdot\vec\nabla (\Id-\alpha^2\Delta)^{-1}\vec u_\alpha \\+\alpha^2\sum_{k=1}^3 &( \partial_k(\Id-\alpha^2\Delta)^{-1}\vec u_\alpha) \cdot\vec\nabla (\partial_k(\Id-\alpha^2\Delta)^{-1}\vec u_\alpha)   -\vec\nabla p_\alpha
\end{split}\end{equation*}   with $\text{ div }\vec u_\alpha=0$ and $\vec u_\alpha(0,.)=\vec u_0$.  
  We can rewrite the equation  as
\begin{equation*}\begin{split}  \partial_t \vec u_\alpha&+((\Id-\alpha^2\Delta)^{-1}\vec u_\alpha).\vec\nabla \vec u_\alpha\\ =&\Delta\vec u_\alpha+  \sum_{k=1}^3 \alpha^2\partial_k \left(   (\partial_k(\Id-\alpha^2\Delta)^{-1} \vec u_{\alpha}))\cdot\vec\nabla (\Id-\alpha^2\Delta)^{-1} \vec u_{\alpha}\right)  -\vec\nabla.  p_\alpha  \end{split}\end{equation*}  
  This is equivalent to write
\begin{equation*}\begin{split}  \partial_t \vec u_\alpha =\Delta\vec u_\alpha &- \mathbb{P}\Div(((\Id-\alpha^2\Delta)^{-1}\vec u_\alpha)\otimes\vec u_\alpha)\\&-  \sum_{k=1}^3  \mathbb{P}\partial_j ((\alpha\partial_k(\Id-\alpha^2\Delta)^{-1} \vec u_{\alpha})\cdot (\alpha\vec\nabla ( \Id-\alpha^2\Delta)^{-1} \vec u_{\alpha} )).  \end{split}\end{equation*}

\item[$\bullet$] 
 The simplified Bardinal model.\\ The simplified Bardina model  is another $\alpha$-model studied by Cao, Lunasin and Titi  in 2006 \cite{CLT06}. This model  is given by
\begin{equation*} \partial_t\vec u_\alpha+ ((\Id-\alpha^2\Delta)^{-1}\vec u_\alpha)\cdot \vec\nabla\,((\Id-\alpha^2\Delta)^{-1} \vec u_\alpha)=   \Delta\vec u_\alpha   -\vec\nabla p_\alpha
\end{equation*}  
where we have again  $\Div\vec u_\alpha=0$ and $\vec u_\alpha(0,.)=\vec u_0$.  
  \begin{equation*}\begin{split}  \partial_t \vec u_\alpha =\Delta\vec u_\alpha &- \mathbb{P}\Div(((\Id-\alpha^2\Delta)^{-1}\vec u_\alpha)\otimes ((\Id-\alpha^2\Delta)^{-1}\vec u_\alpha)).  \end{split}\end{equation*} 
  \end{itemize}

     Thus, when $\vec u_0\in {\rm bmo}^{-1}_0$, all those $\alpha$-models give back the mild solution $\vec u\in X_T$ when $\alpha$ goes to $0$.

\end{document}